\documentclass[11pt,a4paper]{amsart}
\usepackage[a4paper,top=3cm, left=2cm, right=2cm, bottom=2cm]{geometry}
\usepackage{amsfonts}
\usepackage{amsthm}
\usepackage{amsmath}
\usepackage{amssymb}
\usepackage{epic}
\usepackage{verbatim}
\usepackage{mathpazo} 
\usepackage{graphicx}
\usepackage{latexsym}
\usepackage[T1]{fontenc} 
\usepackage[parfill]{parskip} 

\usepackage{hyperref}
\hypersetup{colorlinks=true, urlcolor=blue, citecolor=red, linkcolor=blue}
\usepackage{url}

\newcommand{\A}{\mathcal{A}}

\newcommand{\empt}{\varepsilon}
\newcommand{\NN}{\mathbb{N}}

\newcommand{\rev}{\widetilde}

\newcommand{\bw}{\boldsymbol{w}}

\renewcommand{\Alph}{\mbox{Alph}}

\theoremstyle{plain}
\newtheorem{theorem}{Theorem} 
\newtheorem{lemma}[theorem]{Lemma}
\newtheorem{corollary}[theorem]{Corollary}
\newtheorem{proposition}[theorem]{Proposition}

\newtheorem{remark}[theorem]{Remark}

\theoremstyle{definition}
\newtheorem{definition}[theorem]{Definition}
\newtheorem{example}[theorem]{Example}

\theoremstyle{remark}

\newtheorem*{note}{Note}

\begin{document}

\title{Generalized trapezoidal words}

\author{Amy Glen}

\thanks{Corresponding author: Amy Glen}
\thanks{\textit{Note:} Preliminary work on this new class of words was presented in talks by the first author at the \textit{35th Australasian Conference on Combinatorial Mathematics \& Combinatorial Computing} (35ACCMCC) in December 2011 and at a workshop on \textit{Outstanding Challenges in Combinatorics on Words} at the Banff International Research Station (BIRS) in February 2012. }

\address{Amy Glen \newline 
\indent School of Engineering \& Information Technology \newline
\indent Murdoch University \newline
\indent  90 South Street
\newline
\indent Murdoch WA 6150 AUSTRALIA}%
\email{\href{mailto:amy.glen@gmail.com}{amy.glen@gmail.com}}

\author{Florence Lev\'e}

\address{Florence Lev\'e \newline
\indent Laboratoire MIS \newline
\indent Universit\'{e} de Picardie Jules Verne \newline
\indent 33 rue Saint Leu \newline 
\indent 80039 Amiens FRANCE}%
\email{\href{mailto:florence.leve@u-picardie.fr}{florence.leve@u-picardie.fr}}

\date{Submitted August 2014; Revised February 2015.}

\begin{abstract} The \emph{factor complexity function} $C_w(n)$ of a finite or infinite word $w$ counts the number of distinct \emph{factors} of $w$ of length $n$ for each $n \ge 0$. A finite word $w$ of \emph{length} $|w|$ is said to be \emph{trapezoidal} if the graph of its factor complexity $C_w(n)$ as a function of $n$ (for $0 \leq n \leq |w|$) is that of a regular trapezoid (or possibly an isosceles triangle); that is, $C_w(n)$ increases by 1 with each $n$ on some interval of length~$r$, then $C_w(n)$ is constant on some interval of length $s$, and finally $C_w(n)$ decreases by 1 with each $n$ on an interval of the same length $r$. Necessarily $C_w(1)=2$ (since there is one factor of length $0$, namely the \emph{empty word}), so any trapezoidal word is on a binary alphabet. Trapezoidal words were first introduced by de~Luca (1999) when studying the behaviour of the factor complexity of \emph{finite Sturmian words}, i.e., factors of infinite ``cutting sequences'', obtained by coding the sequence of cuts in an integer lattice over the positive quadrant of $\mathbb{R}^2$ made by a line of irrational slope. Every finite Sturmian word is trapezoidal, but not conversely. However, both families of words (trapezoidal and Sturmian) are special classes of so-called \emph{rich words} (also known as \emph{full words}) -- a wider family of finite and infinite words characterized by containing the maximal number of palindromes -- studied in depth by the first author and others in~2009.

In this paper, we introduce a natural generalization of trapezoidal words over an arbitrary finite alphabet $\mathcal{A}$, called \emph{generalized trapezoidal words} (or \emph{GT-words} for short). In particular, we study combinatorial and structural properties of this new class of words, and we show that, unlike the binary case, not all GT-words are rich in palindromes when $|\mathcal{A}| \geq 3$, but we can describe all those that are rich.
\end{abstract}



\maketitle

\section{Introduction}

Given a finite or infinite word $w$, let $C_w(n)$  (resp.~$P_w(n)$) denote the \textit{factor complexity function} (resp.~the \textit{palindromic complexity function}) of $w$, which associates with each integer $n\geq 0$ the number of distinct factors (resp.~palindromic factors) of $w$ of length $n$. The well-known \textit{infinite Sturmian words} are characterized by both their factor complexity and palindromic complexity. In 1940 Morse and Hedlund~\cite{gHmM40symb} established that an infinite word $\bw$ is Sturmian if and only if $C_{\bw}(n)=n+1$ for each $n\geq 0$. Almost half a century later, in 1999, Droubay and Pirillo~\cite{xDgP99pali} showed that an infinite word $\bw$ is Sturmian if and only if  $P_{\bw}(n)=1$ whenever $n$ is even, and $P_{\bw}(n)=2$ whenever $n$ is odd. In the same year, de~Luca~\cite{aD99onth} studied the factor complexity function of finite words and showed, in particular, that if $w$ is a \textit{finite Sturmian word} (meaning a factor of an infinite Sturmian word), then the graph of $C_w(n)$ as a function of $n$ (for $0\leq n\leq |w|$, where $|w|$ denotes the \textit{length} of $w$) is that of a regular trapezoid (or possibly an isosceles triangle). That is, $C_w(n)$ increases by $1$ with each $n$ on some interval of length $r$, then $C_w(n)$ is constant on some interval of length $s$, and finally $C_w(n)$ decreases by $1$ with each $n$ on an interval of the same size $r$. Such a word is said to be \textit{trapezoidal}. Since $C_w(1)=2$, any  trapezoidal word is necessarily on a binary alphabet.

In this paper, we study combinatorial and structural properties of the following natural generalization of trapezoidal words over an arbitrary finite alphabet.

\pagebreak

\begin{definition} \label{D:trapezoidal} A finite word $w$ with \textit{alphabet} $\Alph(w) := \A$, $|\A| \geq 2$, is said to be a \textbf{\emph{generalized trapezoidal word}} (or \textbf{\emph{GT-word}} for short) if there exist positive integers $m$, $M$ with $m \leq M$ such that the factor complexity funtion $C_w(n)$ of $w$ increases by $1$ for each $n$ in the interval $[1,m]$, is constant for each $n$ in the interval $[m,M]$, and decreases by 1 for each $n$ in the interval $[M,|w|]$. That is, $w$ is a GT-word if there exist positive integers $m$, $M$ with $m \leq M$ such that $C_w$ satisfies the following:
\begin{center}
\begin{tabular}{ll}
$C_w(0)=1$, &~ \\ 
$C_w(i) = |\A| + i - 1$ &~ for $1 \leq i \leq m$, \\
$C_w(i+1)=C_w(i)$     &~ for $m \leq i \leq M-1$, \\
$C_w(i+1)=C_w(i)-1$   &~ for $M \leq i \leq |w|$.\\
\end{tabular}
\end{center}
\end{definition}

So if a finite word $w$ consisting of at least two distinct letters is a GT-word then the graph of its factor complexity $C_w(n)$ as a function of $n$ (for $0 \leq n \leq |w|$) forms a regular trapezoid (or possibly an isosceles triangle when $m=M$) on the interval $[1,|w|-|\A|+1]$, as shown in Figure~\ref{F:trap-graph} below. In the case of a 2-letter alphabet, the GT-words are precisely the (binary) trapezoidal words studied by de~Luca~\cite{aD99onth}. 

\begin{figure}[htb!]
\includegraphics[scale=0.8]{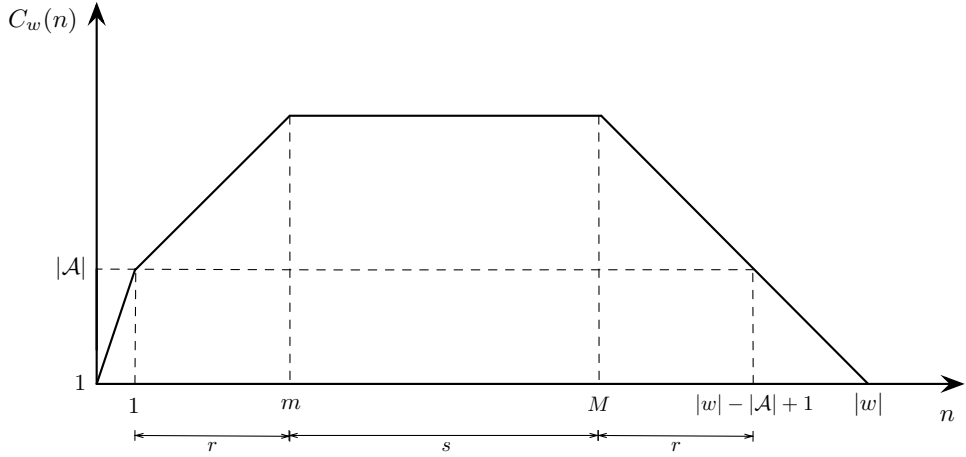}
\caption{{\small Graph of the factor complexity function $C_w(n)$ of a GT-word $w$.}}
\label{F:trap-graph}
\end{figure}

\begin{remark} \label{R:one-letter-words} In what follows, any finite word $w$ with $|\Alph(w)|=1$ (e.g., $w=aaaa$) will also be classed as a generalized trapezoidal word since the complexity function of such a word, being constant on the interval $[1,|w|]$, satisfies the ``trapezoidal conditions'' in Definition~\ref{D:trapezoidal} with $m=1$ and $M=|w|$.
\end{remark} 

Let $w$ be a finite word with alphabet $\Alph(w) := \A$, $|\A| \geq 2$. A \textit{right special} factor $u$ of $w$ is one that can be extended to the right by at least two different letters and still remains a factor of $w$, i.e., $ux$ is a factor of $w$ for at least two different letters $x \in \A$.  Let $R_w$ ($= R$) denote the smallest positive integer $r$ such that $w$ has no right special factor of length $r$ and let $K_w$ ($= K$) denote the length of the shortest \textit{unrepeated suffix} of $w$ (i.e., the shortest suffix of $w$ that occurs only once in $w$). In~\cite[Proposition~4.7]{aD99onth}, de~Luca proved that a finite word $w$ is a binary trapezoidal word if and only if $|w|=R_w+K_w$. 

\begin{example} \label{E:binary-trap}
The binary word $w=aaabb$ of length $5$ has complexity sequence $\{C_w(n)\}_{n\geq0} = 1, 2, 3, 3, 2, 1$, so $w$ is clearly a trapezoidal word and we see that $R_w = 3$, $K_w = 2$, and indeed $|w| = R_w + K_w$.
\end{example}

Every finite Sturmian word $w$, being a binary trapezoidal word, satisfies the condition $|w| = R_w + K_w$, but not conversely, i.e., there exist non-Sturmian binary trapezoidal words. For instance, the binary trapezoidal word $aaabb$ in Example~\ref{E:binary-trap} is not Sturmian because it contains two palindromes, $aa$ and $bb$, of \textit{even} length $2$. All non-Sturmian trapezoidal words were classified by D'Alessandro~\cite{fD02acom} in~2002.

It is natural to wonder if there is an analogous combinatorial characterization in terms of $R$ and $K$ for generalized trapezoidal words. One might guess, for instance, that GT-words are, perhaps, the finite words $w$ that satisfy the condition $|w| = R_w + K_w + |\Alph(w)| - 2$. But whilst it is true that any word satisfying this ``RK-condition'' is a GT-word (see Corollary~\ref{C:RK-condition} later), the converse does not hold. For example, the GT-word $u=ababadac$ of length $8$ with complexity sequence $\{C_u(n)\}_{n\geq0} = 1, 4, 5, 5, 5, 4, 3, 2, 1$ has $R_u=4$ and $K_u=1$, but $R_u+K_u+2 \ne 8$. However, we see that $|u| = R_v + K_v + |\Alph(u)| - 2$ where $v$ is the so-called ``heart'' of $u$ (see Definition~\ref{D:heart}); namely, the factor $v=ababada$ (with $R_v = 4$, $K_v = 2$) obtained from $u$ by deleting its last letter $c$ that occurs only once in $u$. It turns out that the condition $|w| = R_v + K_v + |\Alph(w)| -2$, where $v$ is the \textit{heart} of $w$, does indeed characterize generalized trapezoidal words (see Theorem~\ref{T:characterization1-take2} later).

In the next section, we prove some combinatorial properties of GT-words, particularly with respect to the parameters $R$ and $K$, followed by our main results in Section~\ref{S:main} where we prove some characterizations of GT-words (see Theorem~\ref{T:characterization1-take2}, Corollary~\ref{C:characterization1-take2}, and Theorem~\ref{T:triangle}) and describe all the GT-words that are \textit{rich} in palindromes (see Theorem~\ref{T:GT-rich}). We use standard terminology and notation for combinatorics on words, as in the book~\cite{mL02alge}, for instance.

\section{Preliminary Results} \label{S:preliminary}

The following result of de Luca \cite{aD99onth}, which describes the graph of the complexity function of any given word $w$, will be needed in what follows.

\begin{proposition} \cite[Proposition~4.2]{aD99onth} \label{P:deLuca1} Let $w$ be a finite word with $|\Alph(w)| \geq 2$ and let $m=\min\{R_w,K_w\}$, $M=\max\{R_w,K_w\}$. The factor complexity function $C_w$ of $w$ is strictly increasing on the interval $[0, m]$, is non-decreasing on the interval $[m, M]$, and is strictly decreasing on the interval $[M, |w|]$. Moreover, for all $n \in [M, |w|]$, one has $C_w(n+1)=C_w(n)-1$. If $R_w < K_w$, then $C_w$ is constant on the interval~$[m, M]$. 
\end{proposition}

As a first step towards obtaining a combinatorial characterization of generalized trapezoidal words, we describe, in the following theorem, the finite words $w$ satisfying the so-called \textit{RK-condition} $|w| = R_w + K_w + |\Alph(w)| - 2$.

\begin{theorem} \label{T:RK-characterization} Suppose $w$ is a finite word with $\Alph(w) := \A$, $|\A| \geq 2$. Then $|w| = R_w + K_w + |\A| - 2$ if and only if the factor complexity function $C_w$ of $w$ satisfies: 
\begin{center}
\begin{tabular}{ll}
$C_w(0)=1$, &~ \\ 
$C_w(i) = |\A| + i - 1$ &~for $1 \leq i \leq m$, \\
$C_w(i+1)=C_w(i)$     &~ for $m \leq i \leq M-1$, \\
$C_w(i+1)=C_w(i)-1$   &~ for $M \leq i \leq |w|$, \\
\end{tabular}
\end{center}
where $m=\min\{K_w,R_w\}$ and $M=\max\{K_w,R_w\}$. Moreover, $C_w(R_w) = C_w(K_w)$ and the maximal value of $C_w$ is $m+|\A|-1$.
\end{theorem}
\begin{proof}
($\Rightarrow$): Suppose that $|w|=R_w+K_w+|\A|-2$. 
We first consider the case $R_w \leq K_w$ so that $m=R_w$ and $M=K_w$. By Proposition~\ref{P:deLuca1}, we know that $C_w$ is strictly increasing on the interval $[0,R_w]$ and is constant on the interval $[R_w,K_w]$, and then strictly decreasing on the interval $[K_w,|w|]$; in particular, 
\begin{equation} \label{eq:RK-characterization-1}
C_w(i+1) = C_w(i) - 1 \quad \mbox{for $i \in [K_w, |w|]$}.
\end{equation}
It remains to show that $C_w(i)= |\A| + i -1$ for $i \in [1, R_w]$. Since $C_w(1) = |\A|$, we have $C_w(i) \geq |\A| + i - 1$ for $i \in [1, R_w]$, so $C_w(R_w) \geq |\A| + R_w - 1$, and in fact, we have $C_w(R_w) = |\A| + R_w - 1$. Indeed, from Equation~\eqref{eq:RK-characterization-1} and the hypothesis $|w| = R_w + K_w + |\A| - 2$, we deduce that 
\[
1 = C_w(|w|) = C_w(K_w) - (|w| - K_w) = C_w(K_w) - R_w - |\A| + 2.
\]
Hence, since $C_w(K_w) = C_w(R_w)$, it follows that $C_w(R_w) = |\A| + R_w -1$. This implies that $C_w(i) = |\A| + i - 1$ for $i \in [1,R_w]$. For if not, then there exists an $r < R_w$ such that $C_w(r) >  |\A| + r -1$ and $C_w(r+n) \geq C_w(r) + n$ for $n=1,2, \ldots, R_w-r$ (e.g., see \cite{aD99onth}) implying that 
\[
C_w(R_w) \geq  C_w(r) + R_w - r > (|\A| + r - 1) + R_w - r = |\A| + R_w - 1, 
\]
which is impossible because $C_w(R_w) = |\A| + R_w - 1$.

Now let us consider the case $K_w < R_w$, so that $m=K_w$ and $M=R_w$. By Proposition~\ref{P:deLuca1}, we know that $C_w$ is strictly increasing on the interval $[0,K_w]$ and is non-decreasing on the interval $[K_w, R_w]$, and then strictly decreasing on the interval $[R_w,|w|]$; in particular, 
\begin{equation} \notag \label{eq:RK-characerization-2}
C_w(i+1) = C_w(i) - 1 \quad \mbox{for $i \in [R_w, |w|]$}.
\end{equation}
Moreover, by the hypothesis $|w| = R_w + K_w + |\A| - 2$, it is easy to see that the maximal value of the complexity is $C_w(R_w) = |w| - R_w + 1 = K_w + |\A| - 1$. But since $K_w < R_w$, we have $C_w(K_w) \geq K_w + |\A| - 1$, and therefore, $C_w(R_w) = C_w(K_w) = |\A| + K_w - 1$. Using a similar argument to that used in the previous case, it follows that $C_w(i) = |\A| + i - 1$ for $i \in [1,K_w]$.

($\Leftarrow$): To prove the converse statement, let us first suppose that $K_w < R_w$. Then by assumption, $C_w(i) = |\A| + i - 1$ for $i \in [1, K_w]$, and so $C_w(K_w) = |\A| + K_w - 1$. Moreover, since $C_w(i+1) = C_w(i)$ for $i \in [K_w, R_w -1]$, it follows that $C_w(K_w) = C_w(R_w)$. Furthermore, by the condition $C_w(i+1) = C_w(i) - 1$ for $i \in [R_w, |w|]$, we have $C_w(R_w) = |w| - (R_w - 1) = |w| - R_w + 1$. Hence, $|w| - R_w + 1 = K_w + |\A| - 1$, i.e., $|w| = R_w + K_w + |\A| - 2$. On the other hand, if we suppose now that $K_w \geq R_w$, we easily deduce from the hypothesis that $C_w(R_w) = R_w + |\A| - 1 = C_w(K_w) = |w| - K_w + 1$, and therefore $|w| = R_w + K_w + |\A| - 2$. 
\end{proof}

\begin{corollary} \label{C:RK-condition} Suppose $w$ is a non-empty finite word. If $|w| = R_w + K_w + |\Alph(w)| - 2$, then $w$ is a GT-word. 
\end{corollary}
\begin{proof}
In the case when $|\Alph(w)| \geq 2$, the result follows immediately from Theorem~\ref{T:RK-characterization}. It remains to consider the case when $|\Alph(w)| =1$. Any such word has length equal to $R_w + K_w + |\Alph(w)| - 2$ where $R_w = 1$ and $K_w = |w|$, and clearly all such words are GT-words since the graphs of their complexity functions, being constant, satisfy the trapezoidal property (see Remark~\ref{R:one-letter-words}). 
\end{proof}

It should be noted here that we are using a slightly different convention for $R_w$ compared to  \cite{aD99onth}; in that paper, a power of a letter had $R_w=0$, not $1$, since the empty word was considered to be a factor that is not right special in such a word. However, our convention (defining $R_w$ to be the smallest \textit{positive} integer such that $w$ does not have a right special factor of length $R_w$) allows us to view single-letter words as GT-words satisfying the RK-condition.

Let $w$ be a finite word with $\Alph(w) := \A$, $|\A| \geq 2$. A factor $u$ of $w$ is said to be \textit{right-extendable} if there exists a letter $x \in \A$ such that $ux$ is a factor of $w$ and $u$ is said to have \textit{right-valence} $j \geq 0$ if $u$ is right-extendable in $w$ by $j$ distinct letters (e.g., see \cite{aD99onth}). \textit{Left special factors}, \textit{left-extendable factors}, and \textit{left-valences} are all similarly defined. A factor of $w$ is said to be \textit{bispecial} if it is both left and right special. Clearly, any left (resp.~right) special factor has left (resp.~right) valence at least $2$. In the case when $|\A| = 2$, left and right special factors have valences equal to $2$ since any factor of a word is only extendable to the left or right in at most two different ways.

For the subset of GT-words satisfying the RK-condition, we have the following consequence of Theorem~\ref{T:RK-characterization} and Proposition~\ref{P:deLuca1}.

\begin{proposition} \label{P:right-special-characterization} 
Suppose $w$ is a finite word with $|\Alph(w)| \geq 2$. Then $|w|=R_w+K_w+|\Alph(w)|-2$ if and only if there exists exactly one right special factor of $w$ of length~$i$ for every (non-negative) integer $i < R_w$ and each non-empty right special factor of $w$ has right-valence $2$.
\end{proposition}
\begin{proof}
($\Rightarrow$): Let $|w|=R_w+K_w+|\Alph(w)|-2$, $m=\min\{R_w,K_w\}$, and $M=\min\{R_w,K_w\}$. Suppose $w$ contains two distinct right special factors of the same length $i$ for some $i < R_w$. If $i < m$, then since each factor of $w$ of length less than $m$ is right-extendable in $w$, we would have $C_w(i+1) - C_w(i) > 1$, contradicting the fact that $C_w(n+1) - C_w(n) = 1$ for each $n \in [1,m-1]$ by Theorem~\ref{T:RK-characterization}. On the other hand, if $m \leq i < R_w$ (in which case $m=K_w$ and $M=R_w$), then there being two distinct right special factors of length $i$ implies that $C_w(i+1) - C_w(i) \geq 1$. Indeed, since all factors of length $i$ of $w$, except the suffix of $w$ of length $i$, are right-extendable in $w$, the two different right special factors of length $i$ would increase the complexity of $w$ by at least $1$ when going from factors of length $i$ to factors of length $i+1$. But this contradicts that fact that $C_w$ must be constant on the interval $[K_w, R_w]$ by Theorem~\ref{T:RK-characterization}. So $w$ contains exactly one right special factor of each length $i < R_w$. Furthermore, using similar reasoning as above, we deduce that all non-empty right special factors of $w$ must have right-valence $2$.

\noindent
($\Leftarrow$:) Conversely, if $w$ contains exactly one right special of each length $i < R_w$ with each non-empty right special factor having right-valence $2$, then using Proposition~\ref{P:deLuca1}, we deduce that $C_w$ must satisfy the conditions in Theorem~\ref{T:RK-characterization}, and hence $|w| = R_w + K_w + |\Alph(w)| - 2$.
\end{proof}

Recall that binary trapezoidal words are precisely those that satisfy $|w| = R_w + K_w$ (see \cite{fD02acom, aD99onth}), so Proposition~\ref{P:right-special-characterization} (above) applies to all binary trapezoidal words, but it is not true in general for all GT-words $w$ with $|\Alph(w)| > 2$ (only those that satisfy the RK-condition). For instance, the GT-word $w=abbac$ does not satisfy the RK-condition (since $R_w=2$, $K_w = 1$, and $R_w + K_w + |\Alph(w)| - 2  = 4 \ne |w|$) and it has two different right special factors of length $1$ (namely the letters $a$ and $b$). We also note that GT-words that do not satisfy the RK-condition can also contain right special factors with right-valences greater than $2$ (and also left special factors with left-valences greater than $2$). For example, in the GT-word $u=ababadac$ (which does not satisfy the RK-condition), the letter $a$ is a right special factor with right-valence~$3$. The letters $b$ and $d$ are each right-extendable only by the letter $a$ in $u$, whereas the letter $c$ is not right-extendable as a factor of $u$, so there is only a jump of $1$ between $C_u(1) = 4$ and $C_u(2) = 5$.

As introduced by de~Luca~\cite{aD99onth}, one can define two parameters, $L_w$ and $H_w$, analogous to $R_w$ and $K_w$, when considering left factors instead of right factors. Specifically, for a given finite word $w$ with $|\Alph(w)| \geq 2$, $L_w$ is defined to be the smallest positive integer $\ell$ such that $w$ has no left special factor of length $\ell$ and $H_w$ is the shortest unrepeated prefix of $w$. In \cite[Proposition~4.6]{aD99onth}, de~Luca proved that any finite word $w$ satisfies $|w| \geq R_w + K_w$ and also $|w| \geq L_w + H_w$. Moreover, $C_w(R_w) = C_w(L_w)$ and $\max\{R_w, K_w\} = \max\{L_w, H_w\}$ \cite[Corollary~4.1]{aD99onth}. 

Using the \textit{dual} of Proposition~\ref{P:deLuca1}, in which $R_w$ is replaced by $L_w$ and $K_w$ is replaced by $H_w$ (see \cite{aD99onth}), we can prove (in a symmetric way) the corresponding dual of Theorem~\ref{T:RK-characterization}, and subsequently we have the following dual of Proposition~\ref{P:right-special-characterization}. 

\begin{proposition} \label{P:left-special-characterization} 
Suppose $w$ is a finite word with $|\Alph(w)| \geq 2$. Then $|w|=L_w+H_w+|\Alph(w)|-2$ if and only if there exists exactly one left special factor of $w$ of length~$i$ for every (non-negative) integer $i < L_w$ and each non-empty left special factor of $w$ has left-valence $2$. \qed
\end{proposition}

The following useful result of de Luca~\cite{aD99onth} concerns the structure of right and left special factors of maximal length of a given word $w$ with respect to the parameters $R_w$, $K_w$, $L_w$, and $H_w$.

\begin{proposition} \label{P:deLuca4.5} \cite[Proposition~4.5]{aD99onth} Suppose $w$ is a word and $u$ is a right (left) special factor of $w$ of maximal length. Then $u$ is either a prefix (suffix) of $w$ or bispecial. If $R_w > H_w$ ($L_w > K_w$), then $u$ is bispecial.
\end{proposition}

The next result is a refinement of \cite[Proposition~4.6]{aD99onth}.

\begin{proposition} \label{P:R+K&L+H}
Suppose $w$ is a non-empty finite word. Then
\[
|w| \geq R_w + K_w + |\Alph(w)| - 2 \quad \mbox{and} \quad |w| \geq L_w + H_w + |\Alph(w)| - 2.
\]
\end{proposition}
\begin{proof}
We will prove only the first inequality since the second one is proved by a symmetric argument. If $|\Alph(w)| = 1$, then $R_w=1$ and $K_w = |w|$, so $|w| = R_w + K_w + |\Alph(w)| - 2$. Now suppose that $|\Alph(w)| \geq 2$ and let $m=\min\{R_w,K_w\}$, $M = \max\{R_w, K_w\}$. Since $C_w(1) = |\Alph(w)|$, it follows from Proposition~\ref{P:deLuca1} that $C_w(i) 
\geq |\Alph(w)| + i - 1$ for $i = 1, 2, \ldots, m$, and so $C_w(m) \geq |\Alph(w)| + m - 1$. Moreover, since $C_w(i+1) \geq C_w(i)$ for all integers $i \in [m,M-1]$, we have $C_w(M) \geq C_w(m)$. Furthermore, by Proposition~\ref{P:deLuca1}, we have $C_w(i+1) = C_w(i) - 1$ for all integers~$i \in [M,|w|]$, and therefore $C_w(M) = |w| - (M - 1) = |w| - M + 1$. Hence 
\begin{align*}
|w| &= C_w(M) + M - 1 \\ 
     &\geq C_w(m) + M - 1  \\ 
     &\geq |\Alph(w)| + m - 1 + M - 1 \\
     &= |\Alph(w)| + m + M - 2 \\
     &= |\Alph(w)| + R_w + K_w - 2.
\end{align*}
\end{proof}

Let $|w|_x$ denote the number of occurrences of a letter $x$ in a non-empty finite word $w$.

\begin{definition}[Heart] \label{D:heart} Let $w$ be a non-empty finite word. Suppose that $r$ is the longest (possibly empty) prefix of $w$ such that $|w|_x = 1$ for all $x \in \Alph(r)$ and that $s$ is the longest (possibly empty) suffix of $w$ such that $|w|_x = 1$ for all $x \in \Alph(s)$. If $|w| > |\Alph(w)|$, then there exists a unique (non-empty) word $v$ with  $|v| < |w|$ such that $w=rvs$ and we call $v$ the \textbf{\textit{heart}} of $w$. Otherwise, if $|w| = |\Alph(w)|$, then $r=s=w$ and we define the \textbf{\textit{heart}} of $w$ to be itself.
\end{definition}

\begin{note} Roughly speaking, for any finite word $w$ with $|w| > |\Alph(w)|$ (i.e., for any finite word $w$ that contains at least two occurrences of some letter), the heart of $w$ is defined to be the unique (non-empty) factor of $w$ that remains if we delete the longest prefix and the longest suffix of $w$ that contain letters only occurring once in $w$. 
\end{note}

Equivalently, the \textit{heart} of a non-empty finite word $w$ is the unique (non-empty) factor $v$ of $w$ such that $w=rvs$ where $rv$ is the longest prefix of $w$ such that $K_{rv} > 1$ and $vs$ is the longest suffix of $w$ such that $H_{vs} > 1$, i.e., $\rev{vs}$ is the longest prefix of $\rev{w}$ such that $K_{\rev{vs}} > 1$ (where $\rev{u}$ denotes the \textit{reversal} of a given word~$u$).

\begin{example} Consider the word $w = ebbacbadf$. By deleting from $w$ the longest prefix and the longest suffix that contain letters only occurring once in $w$, we determine that the heart of $w$ is $v=bbacba$. We also observe that $w$ can be expressed as $w=rvs$ with $r=e$, $s=df$, and where $rv=ebbacba$ is the longest prefix of $w$ such that $K_{rv} > 1$ and $vs=bbacbadf$ is the longest suffix of $w$ such that $H_{vs} > 1$. Note that the word $w$ is  a generalized trapezoidal word with $R_w = 3$ and $K_w = 1$, and we see that $|w| = R_v + K_v + |\Alph(w)| - 2$ (where $R_v=2$ and $K_v = 3$); in fact, this condition characterizes GT-words (see Theorem~\ref{T:characterization1-take2} in the next section).
\end{example}

\begin{remark} \label{R:binary-heart}
All binary trapezoidal words are equal to their own hearts, except those of the form $a^nb$ or $ab^n$ where $a$, $b$ are distinct letters and $n>1$ is an integer. Such (binary) trapezoidal words have hearts of the form $x^n$ for some letter $x \in \{a,b\}$. We also note that any word $w$ with $|w|=|\Alph(w)|$, i.e., any word $w$ that is a product of mutually distinct letters, is a GT-word (such a word $w$ has $R_w=1$ and $K_w=1$ and therefore satisfies the RK-condition).
\end{remark}

The next result generalizes \cite[Proposition~4.8]{aD99onth}. 

\begin{proposition} \label{P:RKLH} Let $w$ be a non-empty finite word with heart $v$. If $|w| = R_v + K_v + |\Alph(w)| - 2$, then
\[
\min\{R_v, K_v\} = \min\{L_v, H_v\} \quad \mbox{and} \quad |w| = L_v + H_v + |\Alph(w)| - 2.
\]
\end{proposition}
\begin{proof} We first observe that if $|w| = R_v + K_v + |\Alph(w)| - 2$, then $|v| = R_v + K_v + |\Alph(v)| - 2$ because, by definition of the heart $v$ of $w$, $|w| - |v| = |\Alph(w)| - |\Alph(v)|$, i.e., $|v| = |w| - |\Alph(w)| + |\Alph(v)|$. The result is trivial if $|\Alph(v)| = 1$ because, in this case, $R_v = L_v = 1$ and $K_v = H_v = |v|$. So we will henceforth assume that $|\Alph(v)| \geq 2$.

By \cite[Corollary~4.1]{aD99onth}, $C_v(R_v) = C_v(L_v)$ and $\max\{R_v, K_v\} = \max\{L_v, H_v\}$. Hence, since any finite word $v$ satisfies $|v| \geq L_v + H_v + |\Alph(v)| - 2$ (by Proposition~\ref{P:R+K&L+H}), we deduce that $\min\{L_v, H_v\}~\leq~\min\{R_v, K_v\}$. Indeed, we have
\begin{align*}
|v| &= R_v + K_v + |\Alph(v)| - 2 \\ 
     &= \min\{R_v, K_v\} + \max\{R_v, K_v\} + |\Alph(v)| - 2 \\ 
     &= \min\{R_v, K_v\} + \max\{L_v, H_v\} + |\Alph(v)| - 2 \\ 
     &\geq L_v + H_v + |\Alph(v)| - 2 
\end{align*}
and thus $\min\{L_v, H_v\} \leq \min\{R_v, K_v\}$.

Suppose that $L_v \leq H_v$. Then, by \cite[Corollary~4.1]{aD99onth}, $H_v = \max\{R_v, K_v\}$, and thus, since $\min\{L_v, H_v\}~\leq~\min\{R_v, K_v\}$, we have $L_v \leq \min\{R_v, K_v\}$. Furthermore, since $L_v \leq H_v$, it follows from Theorem~\ref{T:RK-characterization} and the dual statement of Proposition~\ref{P:deLuca1} (with $L_v$ and $H_v$ instead of $R_v$ and $K_v$) that $C_v(L_v) = C_v(H_v) = C_v(R_v) = C_v(K_v) = \min\{R_v, K_v\} + |\Alph(v)| - 1$. This implies that  $\min\{R_v, K_v\} \leq L_v \leq \max\{R_v, K_v\}$ because $C_v(i) \geq |\Alph(v)| + i - 1$ for $i = 1, 2, \ldots, \min\{R_v, K_v\}$. But, as deduced above, $L_v \leq \min\{R_v, K_v\}$; thus we must have $L_v = \min\{R_v, K_v\}$, and hence $\min\{L_v, H_v\} = \min\{R_v, K_v\}$.

Let us now suppose that $H_v < L_v = \max\{L_v, H_v\} (= \max\{R_v,K_v\})$. Then, since $\min\{L_v, H_v\} \leq \min\{R_v, K_v\}$, we have $H_v \leq \min\{R_v, K_v\}$. We wish to show that $H_v = \min\{R_v, K_v\}$. Suppose not, i.e., suppose $H_v < \min\{R_v, K_v\}$. Let $m= \min\{R_v, K_v\}$. Since $v$ is a (generalized trapezoidal) word satisfying the RK-condition, it follows from Proposition~\ref{P:right-special-characterization} that there exists exactly one right special factor of $v$ of length~$i$ for every non-negative integer $i < m$, with each right special factor having right-valence~$2$. Likewise, there must exist exactly one left special factor of $v$ of length~$i$ for every non-negative integer $i < m$ (where $H_v < m \leq L_v$), with each left special factor having left-valence~$2$. But this implies that $C_v(m) = C_v(m-1)$, contradicting the fact that the complexity function of $v$ satisfies $C_v(m)=C_v(m-1)+1$ by Theorem~\ref{T:RK-characterization}. Hence we must have $H_v = \min\{R_v, K_v\}$.

Finally, since we have $\min\{L_v,H_v\} = \min\{R_v, K_v\}$ and also $\max\{L_v, H_v\} =  \max\{R_v, K_v\}$, it easily follows that $|v| = L_v + H_v + |\Alph(v)| - 2$, and hence $|w| = L_v + H_v + |\Alph(w)| - 2$.
\end{proof}

We have the following easy consequence of Proposition~\ref{P:RKLH}.

\begin{corollary} \label{C:L&H} 
Let $w$ be a non-empty finite word with heart $v$. Then $|w| = R_v + K_v + |\Alph(w)| - 2$ if and only if $|w| = L_v + H_v + |\Alph(w)| - 2$. \qed
\end{corollary}

\begin{note} If a word $w$ satisfies the $RK$-condition $|w| = R_w + K_w + |\Alph(w)| - 2$ (with $R_w$ and $K_w$ instead of $R_v$ and $K_v$), it is not necessarily true that $|w| = L_w + H_w + |\Alph(w)| - 2$ (and vice versa). For example, the word $u=abbcc$ of length $|u| = 5$ is a GT-word with complexity sequence $\{C_u(n)\}_{n\geq0} = 1, 3, 4, 3, 2, 1$. This GT-word has $R_u=2$, $K_u=2$, and $R_u+K_u+1 = |u|$. On the other hand, $L_u=2$, $H_u=1$, and $L_u + K_u + 1 \ne |u|$. However, we see that $|u| = R_v + K_v + 1$ and $|u| = L_v + H_v + 1$ where $v = bbcc$ is the heart of $u$ with $R_v = L_v = K_v = H_v = 2$. 
\end{note}

In \cite[Corollary~5.3]{aD99onth}, de Luca proved that the minimal period length $\pi_w$ of any finite word $w$ satisfies $\pi_w \geq R_w + 1$, and moreover, if $\pi_w = R_w + 1$, then $|w| = R_w + K_w$. The following proposition gives a refinement of that result.

\begin{proposition} \label{P:min-period} Suppose $w$ is a non-empty finite word with minimal period length $\pi_w$. Then $\pi_w \geq R_w + |\Alph(w)| - 1$. Moreover, if $\pi_w = R_w + |\Alph(w)| - 1$, then $|w| = R_w + K_w + |\Alph(w)| - 2$.
\end{proposition}
\begin{proof}
For any finite word $w$ with minimal period length $\pi_w$, we have $\pi_w \geq |w| - K_w + 1$  \cite[Proposition~5.1]{aD99onth}, and by Proposition~\ref{P:R+K&L+H}, $R_w \leq |w| - K_w - |\Alph(w)| + 2$. Hence $\pi_w \geq R_w + |\Alph(w)| - 1$. Furthermore, if $\pi_w = R_w + |\Alph(w)| - 1$, then it follows from \cite[Proposition~5.1]{aD99onth} that $|w| \leq R_w + K_w + |\Alph(w)| - 2$, and therefore $|w| = R_w + K_w + |\Alph(w)| - 2$ by Proposition~\ref{P:R+K&L+H}.
\end{proof}

We are now ready to prove our main results about GT-words.

\section{Main Results} \label{S:main}

\subsection{Characterizations of GT-words}

\begin{theorem} \label{T:characterization1-take2} 
Suppose $w$ is a non-empty finite word with heart $v$. Then $w$ is a GT-word if and only if $|w| = R_v + K_v + |\Alph(w)| - 2$.
\end{theorem}

\begin{example} 
The GT-word $w=ababadac$ does not satisfy the $RK$-condition, but it does satisfy the following condition in terms of its heart $v$:
\[
|w| = R_v + K_v + |\Alph(w)| - 2.
\]
Indeed, $v=ababada$ with $R_v = 4$, $K_v = 2$, and $|w| = 4 + 2 + 2 = 8$.
\end{example}

The following lemma is needed for the proof of Theorem~\ref{T:characterization1-take2}.

\begin{lemma} \label{L:K!=1} Suppose $w$ is a non-empty finite word with heart $v$. Then $w$ is a GT-word if and only if $v$ is a GT-word. Moreover, when $w$ (or equivalently $v$) is a GT-word, the maximal interval $[m_w,M_w]$ on which $C_w$ is constant is equal to the maximal interval $[m_v,M_v]$ on which $C_v$ is constant, i.e., $m_w=m_v$ and $M_w=M_v$.
\end{lemma}
\begin{proof}
If $v = w$, we are done. So suppose $v \ne w$. Then $|w| > |\Alph(w)| \geq 2$ and we have $w = rvs$ where $r$ and $s$ are maximal such that $|w|_x = 1$ for all $x \in \Alph(r) \cup \Alph(s)$. We distinguish three cases: $w=vs$ (where $s$ is non-empty), $w = rv$ (where $r$ is non-empty), or $w=rvs$ (where both $r$ and $s$ are non-empty). The latter two cases (when $w = rv$ or $w=rvs$) are proved similarly to the first case (when $w=vs$), so we consider only that case. By definition of $v$ and $s$, there exist distinct letters $a_1, \ldots, a_j \not\in \Alph(v)$ such that $s = a_1\cdots a_j$ (i.e., $w=va_1\cdots a_j$) where $j = |w| - |v|$. Hence, $C_w(i) = C_v(i) + j$ for $i = 1, 2, \ldots, |v|$. 

Now suppose $w$ is a GT-word. Let $[m_w, M_w]$ be the maximal interval on which $C_w$ is constant and let $[m_v,M_v]$ be the maximal interval on which $C_v$ is constant. Then since $C_v(i) = C_w(i) - j$  for $i = 1, 2, \ldots, |v|$, we see that $C_v$ increases by 1 with each $n$ in the interval $[1,m_w]$, is constant on the interval $[m_w, M_w]$, and decreases by 1 with each $n$ in the interval $[M_w,|v|]$. So $v$ is a GT-word, and moreover $m_v = m_w$ and $M_v = M_w$.

Conversely, suppose $v$ is a GT-word. Then since $C_w(i) = C_v(i) +  j$  for $i = 1, 2, \ldots, |v|$, we see that $C_w$ increases by 1 with each $n$ in the interval $[1,m_v]$, is constant on the interval $[m_v, M_v]$, and decreases by 1 with each $n$ in the interval $[M_v,|v|]$. Moreover, by Proposition~\ref{P:deLuca1}, $C_w$ must continue to decrease by $1$ with each $n$ in the interval $[|v|+1, |w|]$. Thus, $w$ is a GT-word and $m_w=m_v$, $M_w=M_v$.
\end{proof}

\begin{proof}[Proof of Theorem~\ref{T:characterization1-take2}] 
($\Leftarrow$): First suppose that $|w|=R_v+K_v+|\Alph(w)|-2$ where $v$ is the heart of $w$. By definition of $v$, $|w| - |v| = |\Alph(w)| - |\Alph(v)|$, so $|v| = R_v + K_v + |\Alph(v)| - 2$. Hence, $v$ is a GT-word by Corollary~\ref{C:RK-condition}, and therefore $w$ is also a GT-word by Lemma~\ref{L:K!=1}.

($\Rightarrow$): Conversely, suppose $w$ is a GT-word. If $|w|=|\Alph(w)|$ (i.e., if $w$ is a product of mutually distinct letters), then $w=v$ (by Definition~\ref{D:heart}) and such a GT-word has $R_v=R_w=1$ and $K_v=R_w=1$; whence $|w|=|v|=R_v+K_v+|\Alph(w)| - 2$. So we will henceforth assume that $|w| > |\Alph(w)|$, in which case $K_v > 1$.

Since $w$ is a GT-word, so too is $v$ by Lemma~\ref{L:K!=1}. It suffices to show that $|v| = R_v + K_v + |\Alph(v)| - 2$ because it follows from this equality that $|w| = R_v + K_v + |\Alph(w)| - 2$ by noting that $|w| - |v| = |\Alph(w)| - |\Alph(v)|$.

If $|\Alph(v)| = 1$, then $R_v=1$ and $K_v=|v|$, and hence $|v| = R_v + K_v + |\Alph(v)| - 2$. Now suppose that $|\Alph(v)| \geq 2$ and let $[m_v,M_v]$ denote the maximal interval on which $C_v$ is constant. If $R_v < K_v$, then by Proposition~\ref{P:deLuca1} (and the fact that $v$ is a GT-word), the complexity $C_v$ of $v$ satisfies the conditions of Theorem~\ref{T:RK-characterization} with $m=m_v=R_v$ and $M=M_v=K_v$; whence $|v|=R_v + K_v + |\Alph(v)|-2$. On the other hand, if $R_v \geq K_v$, then $M_v=R_v$ (by Proposition~\ref{P:deLuca1} again) and we claim that $m_v=K_v$. Suppose not, i.e., suppose that $m_v \ne K_v$. If $m_v < K_v$ then $C_v(m_v) = C_v(m_v+1)$. But since each factor of $v$ of length $m_v$ is right-extendable in $v$, this implies that $v$ has no right special factor of length $m_v$; a contradiction (since $m_v < K_v \leq R_v$). Now suppose $m_v > K_v$. Then we have $C_v(K_v+1) = C_v(K_v) + 1$. Let $s$ denote the suffix of $v$ of length $K_v$. Since all factors of $v$ of length $K_v$ except $s$ are right-extendable in $v$, the equality $C_v(K_v+1) = C_v(K_v) + 1$ implies that $v$ contains either two different right special factors of length $K_v$, each with right-valence $2$, or only one right special factor of length $K_v$ with right-valence $3$. But then so too would any GT-word $w'$ with heart $v$ and $K_{w'}=1$ (e.g., $w'=vx$ where $x \not\in\Alph(v)$). And  since $s$ is right-extendable in any such $w'$, it follows that $C_{w'}(K_v+1) = C_{w'}(K_v) + 2$; a contradiction. Thus $m_v=K_v$, and so by Proposition~\ref{P:deLuca1} (and the fact that $v$ is a GT-word), the complexity $C_v$ of $v$ satisfies the conditions of Theorem~\ref{T:RK-characterization} with $m=m_v=K_v$ and $M=M_v=R_v$. Hence $|v| = R_v + K_v + |\Alph(v)| - 2$.
\end{proof} 

In view of Corollary~\ref{C:L&H}, we have the following straightforward consequence of Theorem~\ref{T:characterization1-take2}.

\begin{corollary} \label{C:characterization1-take2} 
Suppose $w$ is a non-empty finite word with heart $v$. Then $w$ is a GT-word if and only if $|w| = L_v + H_v + |\Alph(w)| - 2$. \qed
\end{corollary}

The next result describes the generalized trapezoidal words $w$ that are ``triangular'' in the sense that the graphs of their factor complexity functions are triangles on the interval {\small $[1,|w| - |\A| + 1]$} where $\A = \Alph(w)$.

\begin{theorem} \label{T:triangle}  Let $w$ be a finite word with $|\Alph(w)| \geq 2$ and suppose $v$ is the heart of~$w$. Then the graph of the factor complexity $C_w(n)$ of $w$ as a function of $n$ $(0 \leq n \leq |w|)$ is a triangle on the interval \linebreak $[1,|w|-|\Alph(w)| + 1]$ if and only if $K_v = R_v$.
\end{theorem} 
\begin{proof} ($\Leftarrow$): If $K_v=R_v$, then the graph of the factor complexity function $C_v$ of $v$ is clearly a triangle on the interval $[1,|v|-|\Alph(v)|+1]$ by Proposition~\ref{P:deLuca1}. Hence, by Lemma~\ref{L:K!=1}, the graph of $C_w$ is a triangle on the interval $[1,|w|-|\Alph(w)|+1]$.

\noindent
($\Rightarrow$): Suppose the graph of $C_w$ is a triangle on the interval $[1,|w|-|\Alph(w)|+1]$. More precisely, suppose $C_w$ increases by 1 with each $n$ in the interval $[1,m]$ and decreases by 1 with each $n$ in the interval $[m,|w|-|\Alph(w)|+1]$ for some positive integer $m$.

Let us first consider the case $R_v \geq K_v$. By Proposition~\ref{P:deLuca1}, $C_v$ increases by 1 with each $n$ in the interval $[1,K_v]$, is non-decreasing on the interval $[K_v,R_v]$, and decreases by 1 with each $n$ in the interval $[R_v,|v|]$. But by Lemma~\ref{L:K!=1}, since the graph of $C_w$ is a triangle, so too is the graph of $C_v$; in particular, $C_v$ increases by 1 with each $n$ in the interval $[1,m]$ and decreases by 1 with each $n$ in the interval $[m,|w|-|\Alph(w)|+1]$ where $m$ is the same as above. Hence $m=K_v=R_v$. Similarly, when considering the case $R_v \leq K_v$, we deduce that $m=K_v=R_v$.
\end{proof}

We end this section by noting that the set of all GT-words is \textit{factorial} (i.e., closed by taking factors) and is also closed under reversal.

\begin{proposition} \label{P:factors-closed} 
The set of all GT-words is closed by factors.
\end{proposition}
\begin{proof} Let $w$ be a GT-word. If $|\Alph(w)| = 1$, then clearly all of its factors are GT-words too (see Remark~\ref{R:one-letter-words}). So let us now assume that $|\Alph(w)| \geq 2$. By Lemma~\ref{L:K!=1}, it suffices to consider the heart $v$ of $w$. Suppose, by way of contradiction, that there exists a factor $u$ of $v$ (and hence $w$) that is not a GT-word. Then the heart $v'$ of $u$ is not a GT-word (by Lemma~\ref{L:K!=1} again), so $|v'| \ne R_{v'} + K_{v'} + |\Alph(v')| - 2$ by Theorem~\ref{T:characterization1-take2} . Hence, by Proposition~\ref{P:right-special-characterization}, $v'$ contains two different right special factors of the same length or a right special factor with right-valence greater than $2$. But then the same would be true for the heart $v$ of $w$, and therefore $|v| \ne R_v + K_v + |\Alph(v)| - 2$ by Proposition~\ref{P:right-special-characterization}, contradicting the fact that $v$ is a GT-word (by Theorem~\ref{T:characterization1-take2}).
\end{proof}

\begin{proposition} \label{P:reversal-closure} A finite word $w$ is a GT-word if and only if $\rev{w}$ is a GT-word.
\end{proposition}
\begin{proof} The result follows immediately from the definition of a GT-word (Definition~\ref{D:trapezoidal}) by observing that, for any word $w$, its reversal $\rev w$ has the same complexity as $w$.
\end{proof}

\begin{example} The word $abbcc$ and its reversal are both GT-words with complexity sequence $\{C(n)\}_{n\geq0} = 1, 3, 4, 3, 2, 1$. We note, however, that $abbcc$ satisfies the RK-condition, but its reversal $ccbba$ does not because it has $K=1$, $R=2$, and $R + K + 1 \ne 5$. So the set of all words satisfying the RK-condition is not closed under reversal, whereas the set of all GT-words is closed under this involution.
\end{example}

\subsection{Palindromically-Rich GT-words}

In the case when $|\A| = 2$, the first author, together with de~Luca and Zamboni~\cite{aDaGlZ08rich}, have shown that the set of all palindromic (binary) trapezoidal words coincides with the set of all Sturmian palindromes. Moreover, if $w$ is a (binary) trapezoidal word, then $w$ contains exactly $|w|+1$ distinct palindromes (including the \textit{empty word} $\empt$). That is, (binary) trapezoidal words (and hence finite Sturmian words) are \textit{rich} in palindromes in the sense that they contain the maximum possible number of distinct palindromic factors since, as shown in \cite{xDjJgP01epis}, a finite word of length $n$ contains at most $n+1$ distinct palindromic factors (including $\empt$). More precisely, a finite word $w$ is said to be \textit{rich} if and only if it has $|w|+1$ distinct palindromic factors, and an infinite word is rich if all of its factors are rich~\cite{aGjJsWlZ09pali}. So, roughly speaking, a finite or infinite word is rich if and only if a new palindrome is introduced at each position in the word (see Proposition~\ref{P:xDjJgP01epis} below). 

The family of rich words was studied in depth by the first author and others in \cite{aGjJsWlZ09pali}. Such words were also independently considered in 2004 by Brlek, Hamel, Nivat, and Reutenauer~\cite{sBsHmNcR04onth} who called them \textit{full words}. Amongst various properties of rich words, we have the following characterization in terms of so-called \textit{complete returns}. Let $u$ be a non-empty factor of a finite or infinite word $w$. A factor of $w$ having exactly two occurrences of $u$, one as a prefix and one as a suffix, is called a \textit{complete return} to $u$ in $w$.

\begin{proposition} \cite{aGjJsWlZ09pali} \label{P:aGjJsWlZ09pali} A finite or infinite word $w$ is rich if and only if, for each non-empty palindromic factor $u$ of $w$, every complete return to $u$ in $w$ is a palindrome. 
\end{proposition}

In short, $w$ is rich if and only if all complete returns to palindromes are palindromes in $w$. We also have the following characterization of rich words that will be useful in what follows. Note that a factor $u$ of a word $w$ is said to be \textit{unioccurrent} in $w$ if $u$ has exactly one occurrence in $w$.

\begin{proposition} \cite[Proposition~3]{xDjJgP01epis} \label{P:xDjJgP01epis} A word $w$ is rich if and only if all of its prefixes have a unioccurrent palindromic suffix (called a \emph{ups} for short). Similarly, $w$ is rich if and only if all of its suffixes have a unioccurrent palindromic prefix.
\end{proposition}

\begin{note} The ups of a rich word $w$ is necessarily the \textit{longest palindromic suffix} of $w$.
\end{note}

Whilst it is true that all binary trapezoidal words are rich~\cite{aDaGlZ08rich}, the converse does not hold; for example, the binary word $aabbaa$ is rich, but not trapezoidal because its complexity jumps by 2 from $C(1)=2$ to $C(2)=4$. Moreover, unlike in the binary case, not all generalized trapezoidal words are rich. For instance, the word $aabca$ is a non-rich GT-word. However, it is easy to see, for instance,  that all GT-words $w$ with $|\Alph(w)|=3$ and $K_w=1$ are rich since any such word is simply a binary GT-word (which is rich) followed by a new letter (its ups). More generally, our next result describes all the rich GT-words.

\begin{theorem} \label{T:GT-rich} Let $w$ be a GT-word and suppose that its heart $v$ has longest palindromic prefix $p$ and longest palindromic suffix $q$. Then $w$ is rich if and only if one of the following conditions is satisfied.
\begin{itemize}
\item[(i)] $p$ and $q$ are \emph{unseparated} in $v$ (i.e., either $v=pq$,  $v=p=q$, or $p$ and $q$ overlap in $v$).
\item[(ii)] $v = pxq$ where $\Alph(p) \cap \Alph(q) \ne \emptyset$ and $x \in \Alph(p) \cup \Alph(q)$.
\item[(iii)] $v = puq$ where $\Alph(p) \cap \Alph(q) = \emptyset$ and $u=u_1Zu_2$ where $u_1$, $u_2$, $Z$ are words (at least one of which is non-empty) such that $\Alph(u_1) \subseteq \Alph(p)$, $\Alph(u_2) \subseteq \Alph(q)$, and $Z$ contains no letters in common with $p$ and $q$.
\end{itemize}
\end{theorem}

\textbf{Examples} 
\begin{itemize}
\item \textbf{Rich GT-word:} ${\underline{abacaba}de}$ has complexity sequence $\{C(n)\}_{n\geq0} = 1, 5, 6, 6, 6, 5, 4, 3, 2, 1$ with heart $v = p = q = abacaba$ (a palindrome).

\item \textbf{Rich GT-word:} ${\underline{ababada}c}$ has complexity sequence $\{C(n)\}_{n\geq0} = 1, 4, 5, 5, 5, 4, 3, 2, 1$ with heart $v=ababada$, $p=ababa$, and $q=ada$ ($p$ and $q$ overlap). 

\item \textbf{Rich GT-word:} $aaabab$ has complexity sequence $\{C(n)\}_{n\geq 0} = 1, 2, 3, 4, 3, 2, 1$ with heart $v=pq$ where $p=aaa$ and $q=bab$ ($v$ is a product of $p$ and $q$).

\item \textbf{Rich GT-word:} $bacabacac$ has complexity sequence $\{C(n)\}_{n\geq 0} = 1, 3, 4, 5, 5, 5, 4, 3, 2, 1$ with heart $v=paq$ where $p=bacab$ and $q=cac$. Here $\Alph(q) \subset \Alph(p)$ and $p$, $q$ are separated by a letter in $\Alph(p)$.

\item \textbf{Rich GT-word:} $acacbcb$ has complexity sequence $\{C(n)\}_{n\geq 0} = 1, 3, 4, 5, 4, 3, 2, 1$ with heart $v=pcq$ where $p=aca$ and $q=bcb$. Here $|\Alph(p) \cup \Alph(q)| > |\Alph(p) \cap \Alph(q)| \geq 1$ and $p$, $q$ are separated by a letter in $\Alph(p) \cap \Alph(q)$.

\item \textbf{Rich GT-word:} $aabcc$ has complexity sequence $\{C(n)\}_{n\geq0} = 1, 3, 4, 3, 2, 1$ with heart $v=puq$ where $p=aa$, $u= b$, and $q=cc$. Here $\Alph(p) \cap \Alph(q) = \emptyset$ and $p$, $q$ are separated by a letter not in $\Alph(p) \cup \Alph(q)$.

\item \textbf{Rich GT-word:} $aaadcbcb$ has complexity sequence $\{C(n)\}_{n\geq 0} = 1, 4, 5, 6, 5, 4, 3, 2, 1$ with heart $v = puq$ where $p=aaa$, $u=dc$, and $q=bcb$. Here $\Alph(p) \cap \Alph(q) = \emptyset$ and $p$, $q$ are separated by the word $dc$ where the letter $d$ is not in $\Alph(p) \cup \Alph(q)$ and $c \in \Alph(q)$.

\item \textbf{Non-Rich GT-word:} ${\underline{ababadb}c}$ has complexity sequence $\{C(n)\}_{n\geq0} = 1, 4, 5, 5, 5, 4, 3, 2, 1$  
with heart $v=pdq$ where $p=ababa$ and $q=b$. Here $\Alph(q) \subset \Alph(p)$ and $p$, $q$ are separated by a letter not in $\Alph(p)$. Note that $v$ ends with a non-palindromic complete return to $b$.

\item \textbf{Non-Rich GT-word:} $abacbab$ has complexity sequence $\{C(n)\}_{n\geq0} = 1, 3, 4, 5, 4, 3, 2, 1$ with heart $v=pcq$ where $p=aba$ and $q=bab$. Here $\Alph(p) = \Alph(q)$ and $p$, $q$ are separated by a letter that they do not contain. Note that $v$ contains non-palindromic complete returns to each of the letters $a$ and $b$.

\item \textbf{Non-Rich GT-word:} $dcdbacdc$ has complexity sequence $\{C(n)\}_{n\geq 0} = 1, 4, 5, 6, 5, 4, 3, 2, 1$ with heart $v=puq$ where $p=dcd$, $u=ba$, and $q=cdc$. Here $\Alph(p) = \Alph(q)$ and $p$, $q$ are separated by a product of two distinct letters, neither of which they contain themselves. Note that $v$ contains non-palindromic complete returns to each of the letters $c$ and $d$.

\item \textbf{Non-Rich GT-word:} $aaaaaadebcad$ has complexity sequence \\ $\{C(n)\}_{n\geq 0} = 1, 5, 6, 7, 7, 7, 7, 6, 5, 4, 3, 2, 1$ with heart $v=puq$ where $p=aaaaaa$, $u=debca$, $q=d$. Here $\Alph(p) \cap \Alph(q) = \emptyset$ and $p$, $q$ are separated by a product of mutually distinct letters, with last letter in $\Alph(p)$ and first letter in $\Alph(q)$. Note that $v$ begins with a non-palindromic complete return to $a$ and also ends with a non-palindromic complete return to~$d$.

\item \textbf{Non-Rich GT-word:} $abcedabceded$ has complexity sequence \\ $\{C(n)\}_{n\geq 0} = 1, 5, 6, 7, 7, 7, 7, 6, 5, 4, 3, 2, 1$ with heart $v=puq$ where $p=a$, $u=bcedabce$, $q=ded$. Here $\Alph(p) \cap \Alph(q) = \emptyset$ and $p$, $q$ are separated by a word containing letters in $\Alph(p)$ and $\Alph(q)$, but beginning with a letter not in $\Alph(p)$. Note that $v$ begins with a non-palindromic complete return to the letter $a$ and also contains non-palindromic complete returns to each of the other letters $b$, $c$, $d$, and $e$.
\end{itemize}

The proof of Theorem~\ref{T:GT-rich} requires several lemmas (Lemmas~\ref{L:rich-heart}--\ref{L:rich-disjoint}, to follow), but first we show that, in the case of a binary alphabet, every trapezoidal word has unseparated $p$ and $q$. This result allows us to give an alternative proof of the fact that all binary trapezoidal words are rich, originally proved in \cite{aDaGlZ08rich} (see Corollary~\ref{C:binary-hearts} below).

\begin{proposition} \label{P:binary-hearts}
Suppose $w$ is a binary trapezoidal word. Then the longest palindromic prefix $p$ and longest palindromic suffix $q$ of $w$ are unseparated in $w$. 
\end{proposition}
\begin{proof} Let $\Alph(w) = \{a,b\}$. We first observe that if $w$ takes the form $x^ny$ or $xy^n$ where $\{x,y\} = \{a,b\}$ and $n \in \mathbb{N}^+$, then $w$ is simply a product of its longest palindromic prefix and longest palindromic suffix. 

Let us now suppose, on the contrary, that $w$ is a binary trapezoidal word (not of the form $x^ny$ or $xy^n$) of minimal length such that its longest palindromic prefix $p$ and longest palindromic suffix $q$ are separated, i.e., such that $w=puq$ where $u$ is a non-empty word. Clearly, $w$ cannot be a palindrome; otherwise $w$ would itself be its longest palindromic prefix and suffix (i.e., $w=p=q$).

We first show that $p\ne q$. Suppose not, i.e., suppose $w=pup$. Then $u$ must be a non-palindromic word because $w$ is not a palindrome. If $p$ has an occurrence in $w$ that is neither a prefix nor a suffix of $w$, then $w$ would begin with a complete return to $p$ (shorter than $w$), which cannot be a palindrome because otherwise $w$ would begin with a palindrome longer than $p$ (its longest palindromic prefix). So $w$ would begin with a non-palindromic complete return to $p$ (shorter than $w$) of the form $pUp$ where $U$ is a non-palindromic word. Let $s=pUp$. Since any factor of $w$ is a (binary) trapezoidal word (by Proposition~\ref{P:factors-closed} or \cite[Proposition~8]{fD02acom}), the proper prefix $s=pUp$ of $w$ is a binary trapezoidal word. But this contradicts the minimality of $w$ as a binary trapezoidal word (of shortest length) having separated longest palindromic prefix and suffix. Therefore $p$ occurs only twice in $w$, as a prefix and as a suffix. That is, $w=pup$ is itself a non-palindromic complete return to $p$. If $p$ were a letter, say $p=a$, then $w=aua$ where $\Alph(u)=\{b\}$ (since $w$ is a complete return to $p$), which implies that $u=b^n$ for some integer $n\geq 1$. However, this is not possible because $u$ is not a palindrome. So $|p| \geq 2$. If $p=aa$, then $w=aauaa$ where $u$ begins and ends with the letter $b$. In this case, the letter $a$ is bispecial, so $b$ cannot be left or right special by Propositions~\ref{P:right-special-characterization} and \ref{P:left-special-characterization}. In particular, $bb$ cannot a factor of $u$ nor can $aa$ because $w$ is a complete return to $aa$. Hence, we deduce that $u= (ba)^nb$ for some integer $n\geq 1$; a contradiction (since $u$ is not a palindrome). Thus $p \ne aa$, and likewise it can be shown that $p \ne bb$. Hence $|p| \geq 3$. Let us write $p=xPx$ where $x \in \{a,b\}$ and $P$ is a palindrome. Then we have $w=xPxuxPx$. Let $w'$ be the prefix of $w$ obtained from $w$ by removing its last letter $x$, i.e., $w'=wx^{-1} = xPxuxP=puxP$ and let $S=PxuxP$ so that $w'=xS$ and $w=w'x=xSx$. Clearly $p=xPx$ is the longest palindromic prefix of $w'$. Let $q'$ denote the longest palindromic suffix of $w'$. By the minimality of $w$ as binary trapezoidal word (of shortest length) having separated longest palindromic prefix and suffix, $p$ and $q'$ cannot be separated in $w'$. Hence $|uxP| \leq |q'| < |w'|$. If $|q'|=|uxP|$, then $w'$ (and hence $w$) begins with the palindrome $xPxPx$ longer than $p$. Therefore $|uxP| < |q'| < |w'|$ and the prefix $Px$ of $q'$ overlaps with a suffix of the prefix $xPx$ of $w'$. Hence $w'$ either begins with the palindrome $xPPx$ (if $|q'| = |uxP| + 1$) or begins with a palindrome $Q$ having $xP$ as a prefix (and $Px$ as a suffix) with $|xPx| < |Q| < |2P| + 2$ (since $P$, being a palindrome, overlaps with itself in a palindrome). Neither case is possible because $p=xPx$ is the longest palindromic prefix of $w'$. Thus $p \ne q$.

We now claim that $p$ and $q$ are both unioccurrent in $w$. Indeed, if $p$ occurs more than once in $w$, then $w$ would begin with a complete return to $p$ (shorter than $w$), which cannot be a palindrome because otherwise $w$ would begin with a palindrome longer than $p$ (its longest palindromic prefix). So $w$ begins with a non-palindromic complete return to $p$ (shorter than $w$) of the form $pUp$ with $p$ being its longest palindromic prefix and suffix. But since $pUp$ is trapezoidal (by Proposition~\ref{P:factors-closed} or \cite[Proposition~8]{fD02acom}), this contradicts the minimality of $w$ as a binary trapezoidal word (of shortest length) having separated longest palindromic prefix and suffix. Likewise, we deduce that $q$ occurs only once (as a suffix) in $w$. So $H_w \leq |p|$ and $K_w \leq |q|$. Moreover, since $w$ is a binary trapezoidal word, we have $|w| = R_w + K_w$ and $|w| = L_w + H_w$.

If $R_w > H_w$, then by Proposition~\ref{P:deLuca4.5}, the right special factor $r$ of $w$ of maximal length $|r| = R_w - 1$ is bispecial, so $r$ occurs at least twice in $w$. In particular, $w$ contains either both $ara$ and $brb$ or both $arb$ and $bra$. Let $P$ be the shortest prefix of $w$ that contains both $arx$ and $bry$ where $\{x,y\} = \{a,b\}$. Without loss of generality, we may assume that $P$ ends with a factor $S$ of the form $S= arxS' = S''bry$. In particular, $S=aEy$ where $E$ is a complete return to $r$. Let $p'$ (resp.~$q'$) denote the longest palindromic prefix (resp.~suffix) of $P=P'S=P'aEy$. If $|q'| > |S| = |aEy|$, then $P$ would have a proper prefix containing $y\rev{E}a$, of which $y\rev{r}b$ and $x\rev{r}a$ are factors. So $\rev{r}$ would be a bispecial factor of $P$ (and hence of $w$) of length $|r|$, and therefore $\rev{r} = r$ by Propositions~\ref{P:right-special-characterization} and \ref{P:left-special-characterization}. Moreover, we deduce that $E$ must be palindromic complete return to (the palindrome) $r$; otherwise, reasoning as above, we would reach a contradiction. It follows that $y=a$ and $x=b$, and so we see that $w$ has a prefix shorter than $P$ that contains both $arx=arb$ and $bry=bra$, contradicting the minimality of $P$. Thus $|q'| \leq |aEy|$. If $|ry| < |q'| < |aEy|$, then $E$ contains $y\rev{r}$, which implies $\rev{r} \ne r$ (since $E$ is a complete return to $r$). Furthermore, $E$ cannot contain $x\rev{r}$; otherwise $\rev{r}$ would be a left special factor of length $|r|$ different from $r$, which is not possible by Proposition~\ref{P:left-special-characterization}. So $E$ contains a factor $X$ of the form $X=rxX'=X''y\rev{r}$ that contains only one occurrence of $r$ (as a prefix) and only one occurrence of $\rev{r}$ (as a suffix). Such a factor cannot be a palindrome (since $x\ne y$) and it follows that the longest palindromic prefix and longest palindromic suffix of $X$ each have length less than $|r|$ and therefore they must be the same palindrome, $Q$ say. But then $X$ would be a non-palindromic complete return to $Q$, contradicting the minimality of $w$. So either $|q'| \leq |ry|$ or $|q'| = |S| = |aEy|$. In the former case, one can use similar reasoning as above to deduce that $p'$ and $q'$ are separated in $P$, again contradicting the minimality of $w$. In the latter case (when $|q'| = |S| = |aEy|$), $S=aEy$ is a palindrome, which implies that $y=a$ (and hence $x=b$) and that $\rev{r} = r$ (i.e., $r$ is a palindrome). So $w=P'aEaQ'$ where $E$ is a palindromic complete return to $r$ beginning with $rb$ and ending with $br$. Reasoning as above, we deduce that $|p| < |P'ar|$ and $|q| < |raQ'|$ (for if not, we reach a contradiction to either the minimality of $P$ or the fact that $E$ is a complete return to $r$). Let $w'$ denote the prefix of $w$ obtained from $w$ by removing the last letter of $w$, i.e., let $w' = wz^{-1} = P'aEaQ'z^{-1}$ where $z$ is the last letter of $w$. Then, since $|p| < |P'ar|$, we see that $p$ (the longest palindromic prefix of $w$) is also the longest palindromic prefix of $w'$, and using similar arguments as above, we deduce that the longest palindromic suffix $q''$ of $w'$ has length $|q''| < |raQ'| - 1$. Hence, just as $p$ and $q$ are separated in $w$, we see that the longest palindromic prefix and longest palindromic suffix of $w'$ (namely $p$ and $q''$) are separated in $w'$; a contradiction (yet again) to the minimality of $w$ as a binary trapezoidal word (of shortest length) having separated longest palindromic prefix and suffix. Thus $R_w \leq H_w$, and similarly we deduce that $L_w \leq K_w$. 

If $R_w = H_w$, then $|w| = |p| + |u| + |q| \geq R_w + |u| + K_w > R_w + K_w$ (because $R_w = H_w \leq |p|$, $K_w \leq |q|$, and $|u| \geq 1$); a contradiction. So $R_w < H_w$, and using similar reasoning, we deduce that $L_w < K_w$ too. Hence it follows that $R_w = L_w$ and $K_w = H_w$ since $\max\{R_w, K_w\} = \max\{L_w, H_w\}$ and $\min\{R_w,K_w\} = \min\{L_w, H_w\}$ by \cite[Corollary 4.1]{aD99onth} and Proposition~\ref{P:RKLH}. This implies that $|pu| \leq R_w$ since $|pu| = |w| - |q|$ where $|w| = R_w + K_w$ and $K_w \leq |q|$. However, by Proposition~\ref{P:deLuca4.5}, the right special factor $r$ of $w$ of maximal length $|r| = R_w - 1$ is a prefix of $w$, which occurs at least twice in $w$, and therefore $p$ (which is a prefix of $r$ since $|pu| \leq R_w = |r| + 1$) occurs more than once in $w$; a contradiction. Thus $p$ and $q$ are unseparated in $w$.
\end{proof}


\begin{corollary} \label{C:binary-hearts}
All binary trapezoidal words are rich.
\end{corollary}
\begin{proof} If $w$ is a binary trapezoidal word, but not rich, then $w$ contains a non-palindromic complete return to a palindrome $p$ of the form $pup$ where $u$ is a non-palindromic word. Let $r=pup$. Clearly $p$ is both the longest palindromic prefix and the longest palindromic suffix of $r$ and these occurrences of $p$ in $r$ are separated. But this yields a contradiction because $r$, being a factor of $w$,  must be a binary trapezoidal word (by Proposition~\ref{P:factors-closed} or \cite[Proposition~8]{fD02acom}), and therefore its longest palindromic prefix and suffix cannot be separated (by Proposition~\ref{P:binary-hearts}). 
\end{proof}

\begin{lemma} \label{L:rich-heart}
A GT-word $w$ is rich if and only if the heart of $w$ is rich.
\end{lemma}
\begin{proof} Let $w$ be a GT-word. If $w$ is rich, then clearly the heart of $w$ is rich too since any factor of a rich word is rich.

Conversely, suppose that the heart $v$ of $w$ is rich. If $v=w$, we are done. So suppose that $v \ne w$. Then $|w| > |\Alph(w)| \geq 2$ and we have $w=rvs$ where $r$ and $s$ are maximal such that $|w|_x=1$ for all $x \in \Alph(r) \cup \Alph(s)$. Moreover, since $v \ne w$, at least one of the factors $r$ and $s$ is non-empty. If $r$ is non-empty, then $|r| = |\Alph(r)|$ and it is easy to see that each prefix of $r$ has a unioccurrent palindromic suffix -- just the last letter of each prefix. Likewise, if $s$ is non-empty, each prefix of $s$ has a ups. Moreover, since $v$ is rich, each prefix of $v$ has a ups, by Proposition~\ref{P:xDjJgP01epis}. Thus, since the alphabets $\Alph(v)$, $\Alph(r)$, and $\Alph(s)$ are mutually disjoint sets, we see that each prefix of $w=rvs$ has a ups, and thus $w$ is rich, by Proposition~\ref{P:xDjJgP01epis} again.
\end{proof}

\begin{lemma} \label{L:GT-rich} Let $w$ be a GT-word and suppose that its heart $v$ has longest palindromic prefix $p$ and longest palindromic suffix $q$. If $v = puq$ for some non-empty word $u$, then one of the following conditions holds:
\begin{itemize}
\item[(i)] $p=(ax)^na$, $u=x$, $q=(bx)^mb$ for some $n, m \in \mathbb{N}^+$ where $a$, $b$, and $x$ are mutually distinct letters; 
\item[(ii)] $\Alph(p) \subseteq \Alph(q)$ or $\Alph(q) \subseteq \Alph(p)$;
\item[(iii)] $\Alph(p) \cap \Alph(q) = \emptyset$.
\end{itemize}
Moreover, if condition \emph{(i)} is satisfied, then $w$ is rich.
\end{lemma}
In other words, the above lemma says that if $v$ is the heart of a GT-word with \textit{separated} $p$ and $q$ such that the alphabets of $p$ and $q$ are not disjoint and neither alphabet is a subset of the other, then $v$ takes the form given by condition~(i). All other hearts of the form $v=puq$ ($u \ne \empt$) are such that the alphabets of $p$ and $q$ are disjoint or one is a subset of the other.
\begin{remark}
For a given letter $x$, let $\varphi_x$ denote the morphisms defined as follows:
\[
\varphi_x: \begin{cases} x \mapsto x \\ y \mapsto yx \end{cases} \quad \mbox{for all letters $y \ne x$}, 
\]
i.e., $\varphi_x$ maps the letter $x$ to itself and inserts the letter $x$ to the \textit{right} of all other letters different from~$x$. 
A word $v$ satisfying condition (i) in Lemma~\ref{L:GT-rich} can be expressed as $v=\varphi_x(a^{n+1}b^m)b$, i.e., $v$ is obtained from the binary trapezoidal word $a^{n+1}b^{m+1}$ by inserting the letter $x \not\in \{a,b\}$ to the right of each letter except the last. The shortest such word takes the form $\varphi_x(aab)b$. 
\end{remark}

\begin{proof}[Proof of Lemma~\ref{L:GT-rich}]
We first observe that, for any word $v$ with longest palindromic prefix $p$ and longest palindromic suffix $q$, the alphabets of $p$ and $q$ satisfy one of the following three conditions.
\begin{itemize}
\item[(A)] $\Alph(p) \cap \Alph(q) \ne \emptyset$ and neither alphabet is a subset of the other.
\item[(B)] $\Alph(p) \subseteq \Alph(q)$ or $\Alph(q) \subseteq \Alph(p)$ (i.e., one alphabet is a subset of the other).
\item[(C)] $\Alph(p) \cap \Alph(q) = \emptyset$.
\end{itemize}
Furthermore, we observe that there exist hearts $v$ of GT-words of the form $v=puq$ ($u\ne \empt$) with $p$ and $q$ satisfying these conditions. For example, the GT-word $acacbcb$ has heart $v=acacbcb=puq$ (itself) where $p=aca$, $u=c$, $q=bcb$ with the alphabets of $p$ and $q$ satisfying condition (A). Condition (B) is satisfied, for instance, by the GT-word $ababadbc$ which has heart $v=ababadb = puq$ where $p=ababa$, $u=d$, $q=b$ with $\Alph(q) \subset \Alph(p)$. We also observe that the reversal of $v=ababadb$ (which is also a GT-word by Proposition~\ref{P:reversal-closure}) is such that the alphabet of its longest palindromic prefix (namely $b$) is contained in the alphabet of its longest palindromic suffix (namely $ababa$). An example of a heart of a GT-word of the form $v=puq$ satisfying condition (B) with $\Alph(p) = \Alph(q)$ is $abacbab$. And lastly, condition (C) is satisfied, for example, by the GT-word $aabcc$, which has heart $v=aabcc=puq$ (itself) where $p=aa$, $u=b$, $q=cc$. 

Now suppose, more generally, that $v$ is the heart of a GT-word of the form $v=puq$ where $u$ is a non-empty word. Then $|\Alph(v)| \geq 3$. Indeed, if $|\Alph(v)| = 1$, then $v=x^n$ for some letter $x$ and positive integer $n$, in which case $v=p=q$ (i.e., $p$ and $q$ are unseparated). And if $|\Alph(v)| = 2$ then $p$ and $q$ cannot be separated by a non-empty word in $v$ by Proposition~\ref{P:binary-hearts}. We also note that if $|w|=|\Alph(w)|$, then $w=v$ and such a word is simply a product of mutually distinct letters with its first and last letters being its longest palindromic prefix and longest palindromic suffix, respectively, and therefore it satisfies condition (iii) in the statement of the lemma. So we will henceforth assume that $|w| > |\Alph(w)|$, in which case the heart $v$ of $w$ satisfies $K_v > 1$ and $H_v > 1$. This means, in particular, that each letter in $\Alph(v)$ is both left and right extendable in $v$.

We showed above that there exist hearts $v$ of GT-words of the form $v=puq$ ($u \ne \empt$) with $p$ and $q$ satisfying conditions (ii) and (iii) in the statement of the lemma (which correspond to conditions (B) and (C) above). So let us now suppose that $v$ satisfies neither of those two conditions. Then $v$ satisfies condition (A) above, i.e., $\Alph(p) \cap \Alph(q) \ne \emptyset$ and there exist at least two distinct letters $a$ and $c$ such that $a \in \Alph(p)\setminus \Alph(q)$ and $c \in \Alph(q) \setminus \Alph(p)$. We now show that $p$, $q$, and $u$ must take the form given in condition (i) of the lemma.

Let us first note that the palindrome $p$ (resp.~$q$) must contain at least two distinct letters -- at least one letter in $\Alph(p) \cap \Alph(q)$ and at least one in $\Alph(p) \setminus \Alph(q)$ (resp.~$\Alph(q) \setminus \Alph(p)$). Hence the palindromes $p$ and $q$ each have length at least three. Now, $p$ contains a letter $x \in \Alph(q)$ such that $xa$ is a factor of $v$ for some letter $a \in \Alph(p) \setminus \Alph(q)$. Since $q$ is a palindrome, the letter $x$ will occur at least twice in $q$, except if $q$ has odd length and $x$ occurs only once as its central letter. In any case, $x$ will be right-extendable in $q$ by a letter $b \ne a$ (since $a \not\in \Alph(q)$). So $xa$ and $xb$ are factors of $v$ (where $a \ne x$ and $a \ne b$), and therefore $x$ is a right special factor of $v$ of length $1$. On the other hand, $q$ contains a letter $y \in \Alph(p)$ such that $yc$ is a factor of $v$ for some letter $c \in \Alph(q) \setminus \Alph(p)$. Since $p$ is a palindrome, the letter $y$ will occur at least twice in $p$, except if $p$ has odd length and $y$ occurs only once as its central letter. In any case, $y$ will be right-extendable in $p$ by a letter $d \ne c$ (since $c \not\in \Alph(p)$). So $yc$ and $yd$ are factors of $v$ (where $c \ne y$ and $c \ne d$), and therefore $y$ is a right special factor of $v$ of length $1$. We have thus shown that the letters $x$ and $y$ are right special factors of $v$. But since $v$ is a GT-word (with $|v| = R_v + K_v + |\Alph(v)| - 2$ by Theorem~\ref{T:characterization1-take2}), it contains at most one right special factor of each length (see Proposition~\ref{P:right-special-characterization}), and so we must have $x=y$. Therefore $xa$, $xb$, $xc$, and $xd$ are all factors of $v$ where $a\ne b$, $c\ne d$, and moreover, $a\ne c$ since $a \in \Alph(p)\setminus \Alph(q)$ and $c \in \Alph(q) \setminus \Alph(p)$. If $a \ne d$ or $b \ne c$, then $x$ would have right-valence at least 3 in $v$, which is not possible by Proposition~\ref{P:right-special-characterization}. (Alternatively, we observe that since each distinct letter in $v$ is right-extendable in $v$ because $K_v > 1$, the letter $x$ having right-valence at least $3$ would imply that  $C_v(2) \geq |\Alph(v)| + 2$, i.e., $C_v(2) \geq C_v(1) + 2$, contradicting the hypothesis that $v$ is a GT-word.) Hence we must have $a=d$ and $b=c$. It follows that the only factors of $p$ of length $2$ are $xa$ and $ax$ and the only factors of $q$ of length $2$ are $xb$ and $bx$. So we deduce that $p$ takes the form $p=(xa)^nx$ or $p=(ax)^na$ for some $n \in \mathbb{N}^+$ and that $q$ takes the form $q = (xb)^mx$ or $q=(bx)^mb$ for some $m \in \mathbb{N}^+$. Moreover, we note that $xa$, $ax$, $bx$, and $xb$ are the only factors of length $2$ of $v$. Indeed, $x$ must always be preceded or followed by either $a$ or $b$, otherwise $x$ would have right or left valence greater than $2$, and also each of the letters $a$ and $b$ must always be preceded or followed by the letter $x$ in $v$, otherwise we would have another left or right special factor of length $1$ in addition to $x$ (which is bispecial), contradicting the hypothesis that $v$ is a GT-word. So the word $u$ that separates $p$ and $q$ in $v=puq$ cannot contain any letters different from $a$, $b$, or $x$, and any factor of $u$ of length $2$ (if $|u|\geq 2$) must be one of the words $ax$, $xa$, $bx$, $xb$.

We now separately consider the two possible forms of each of $p$ and $q$ with respect to the ``separating word'' $u$.

\textbf{Case 1:} $p=(ax)^na$ for some $n \in \mathbb{N}^+$. In this case, if $u$ begins with the letter $a$, then $a$ would be a right special factor of length $1$; a contradiction. So $u$ must begin with the letter $x$. If $u=x$, then $v$ satisfies condition (i) in the statement of lemma. Otherwise, if $|u| \geq 2$, then $u$ must begin with $xa$ or $xb$ since $x$ must be followed by either $a$ or $b$ in $v$. But if $u$ begins with $xa$, then $v$ begins with the palindrome $(ax)^naxa$ longer than $p$. Thus $u$ begins with $xb$.

\textbf{Case 2:} $p=(xa)^nx$ for some $n \in \mathbb{N}^+$. In this case, $u$ must begin with the letter $a$ or $b$ since $x$ can only be followed and preceded by one of these letters. If $u$ begins with the the letter $a$ (which must always be followed by the letter $x$), then $v$ would begin with the palindrome $(xa)^nxax = (xa)^{n+1}x$ longer than $p$, which is not possible. So $u$ must begin with the letter $b$. If $u=b$, then $q$ must begin with $x$ (not $b$), and so $q$ takes the form $q=(xb)^mx$ for some $m \in \mathbb{N}^+$, in which case $v$ would end with the palindrome $xb(xb)^mx$ longer than $q$. So $|u| \geq 2$ and $u$ begins with $bx$.

\textbf{Case 3:} $q=(bx)^mb$ for some $m \in \mathbb{N}^+$. Using similar reasoning as in Case 1, we deduce that either $v$ satisfies condition (i) or that $u$ must end with $ax$.

\textbf{Case 4:} $q=(xb)^mx$ for some $m \in \mathbb{N}^+$. Using similar reasoning as in Case 2, we deduce that $u$ must end with $xa$.

From the above considerations, we see that either $v$ satisfies condition (i), or $u$ takes one of the following forms:
\[
u = xbu'ax, \quad u = xbu'xa, \quad u = bxu'ax, \quad u = bxu'xa, \quad \mbox{or} \quad u=bxa.
\]
But if $u$ takes any of the above forms, we see that $v$ contains the following $6$ factors of length $3$: $axa$, $xax$, $axb$, $bxa$, $bxb$, $xbx$. This means that the complexity of $v$ jumps by $2$ from $C_v(2) = 4$ to $C_v(3) = 6$, contradicting the hypothesis that $v$ is a GT-word. 

Hence we have shown that  $v$ must satisfy one of the conditions (i)--(iii).

Lastly, it is easily verified that if $v$ satisfies condition (i), then each prefix of $v$ has a unioccurrent palindromic suffix, and hence $v$ is rich by Proposition~\ref{P:xDjJgP01epis}. Thus $w$ is rich by Lemma~\ref{L:rich-heart}.
\end{proof}

\begin{lemma} \label{L:sep-by-x}
Let $w$ be a GT-word with heart $v$ having longest palindromic prefix $p$ and longest palindromic suffix $q$. Suppose $v= pxq$ where $x$ is a letter and $\Alph(p) \subseteq \Alph(q)$ or $\Alph(q) \subseteq \Alph(p)$. Then $w$ is rich if and only if $x \in \Alph(p) \cup \Alph(q)$. 
\end{lemma}
\begin{proof} 
We first note that $|w| > |\Alph(w)|$ because if $|w| = |\Alph(w)|$, then $w=v$ and $v$ would be a product of mutually distinct letters which does not take the assumed form. Hence $K_v > 1$ and $H_v > 1$. It is also easy to see that $p\ne q$ and $|\Alph(v)| \ne 1$. Furthermore, if $|\Alph(v)| =2$ then $p$ and $q$ cannot be separated by a non-empty word in $v$ (by Proposition~\ref{P:binary-hearts}). Hence $|\Alph(v)| \geq 3$. 

($\Rightarrow$): Suppose $w$ is rich, then $v$ is rich (by Lemma~\ref{L:rich-heart}). By way of contradiction, let us suppose that $x \not\in \Alph(p) \cup \Alph(q)$. Without loss of generality, we will assume that $\Alph(q) \subseteq \Alph(p)$. Since $p$ and $q$ contain at least one letter in common, there exists a letter $a \ne x$ such that $p=p'ap''$ and $q=q''aq'$ where $p''$ and $q''$ are (possibly empty) words that contain no letters in common. But $ap''xq''a$ is a complete return to $a$, so it must be a palindrome (since $v$ is rich), which implies that $p'' = q'' = \empt$. So $v=p'axaq'$ where $p' \ne \rev{q'}$ because $v$ is not a palindrome. Moreover, both $p'$ and $q'$ must be non-empty. Indeed, if $p'=\empt$, then $v=axaq'$ and we see that $v$ begins with the palindrome $axa$ which is longer than $p=a$; a contradiction. Likewise, if $q'=\empt$, we reach a contradiction. Now, since both $p'$ and $q'$ are non-empty, there exists a letter $y \in \Alph(p')$ and a letter $z \in \Alph(q')$ (with $y \ne x$ and $z \ne x$) such that $ya$ and $az$ are factors of $v$. But then, since $xa$ and $ax$ are also factors of $v$, we see that the letter $a$ is a bispecial factor of $v$ of length $1$. Moreover, since $K_v > 1$, each letter in $\Alph(v)$ is right-extendable in $v$, and therefore the letter $a$ must always be followed by either $x$ or $z$ in $v$. For if not, then $a$ would have right-valence greater than $2$ in $v$, which would imply that $C_v(2) \geq |\Alph(v)| + 2$, i.e., $C_v(2) \geq C_v(1) + 2$, contradicting the hypothesis that $v$ is a GT-word (also see Proposition~\ref{P:right-special-characterization}). Likewise, the letter $a$ can only be preceded by the letter $x$ or $y$ in $v$ by Proposition~\ref{P:left-special-characterization}. Since $p$ is a palindrome containing $ya$, we see that $ay$ is a factor of $p$. Similarly, since $q$ is a palindrome containing $az$, we see that $za$ must be a factor of $q$. We thus deduce that $y=z$. So the only factors of $v$ of length $2$ are $za$, $az$, $ax$, and $xa$. If $z=a$, then $p=a^m$ and $q=a^n$ for some $m, n \in \mathbb{N}^+$, in which case $v=a^mxa^n$, contradicting the hypothesis that the longest palindromic prefix and longest palindromic suffix of $v$ are separated. So $z \ne a$ and we deduce that $p$ takes the form $p=(az)^ma$ for some $m \in \mathbb{N}^+$ and $q$ takes the form $q=(az)^na$ for some $n \in \mathbb{N}^+$. But then $v=(az)^m a x (az)^n a$, in which case either $v$ is a palindrome (if $m=n$) or the longest palindromic prefix and longest palindromic suffix of $v$ overlap in $v$; a contradiction. Hence $x \in \Alph(p) \cup \Alph(q)$. 

($\Leftarrow$): Conversely, suppose $x \in \Alph(p) \cup \Alph(q)$. We will show that $v$ (and hence $w$) is rich by determining the possible forms of $p$ and $q$ in $v$. Without loss of generality, we will assume that $\Alph(q) \subseteq \Alph(p)$. Then $\Alph(p) = \Alph(v)$, and so $p$ contains at least three distinct letters (since $|\Alph(v)| \geq 3$); in particular, $p$ contains the letter $x$ and at least two other distinct letters $a$ and $b$ (different from $x$). We first observe that $v$ cannot contain the square of a letter. Indeed, if $yy$ were a factor of $v$ for some letter $y$, then $y$ would be a factor of $p$ (since $\Alph(p) = \Alph(q)$), and therefore $y$ would be preceded and followed by some letter $z \ne y$ (since $|\Alph(p)| \geq 3$). But then $yy$ would be bispecial and $y$ can only ever be preceded (and followed) by itself and $z$ in $v$, otherwise $y$ would have left or right valence more than $2$, which is impossible by Propositions~\ref{P:right-special-characterization} and \ref{P:left-special-characterization}. Likewise, the letter $z$ can only ever be preceded and followed by the letter $y$ in $v$, which implies that $v$ is a binary word; a contradiction. So $v$ does not contain the square of a letter. Moreover, using similar arguments as in the proof of Lemma~\ref{L:GT-rich} (with the aid of Propositions~\ref{P:right-special-characterization}--\ref{P:left-special-characterization} again), we deduce that $p$ (and hence $v$) contains only two distinct letters different from $x$ and that $p$ and $q$ must take the following forms:
\[
p=[b (xa)^m x]^k b \quad \mbox{and} \quad q = (ax)^n a
\]
where $a$ and $b$ are distinct letters (different from $x$) and $k, m, n \in \NN^+$ with $n \geq m$. That is,
\[
v=pxq=[b(xa)^mx]^kb(xa)^{n+1}
\]
with $|v| = |p| + |q| + 1 = R_v + K_v + |\Alph(v)|-2$ where $R_v=|p|$ and $K_v=|q|$.  
(Note that when $\Alph(p) \subseteq \Alph(q)$ the forms of $p$ and $q$ are opposite to those given above.) It is easy to see that every prefix of any such $v$ has a unioccurrent palindromic suffix. Hence $v$ is rich (by Proposition~\ref{P:aGjJsWlZ09pali}), and thus $w$ is rich by Lemma~\ref{L:rich-heart}.
\end{proof}

\begin{lemma} \label{L:sep-by-u}
Let $w$ be a GT-word with heart $v$ having longest palindromic prefix $p$ and longest palindromic suffix $q$. If $v= puq$ where $|u|\geq 2$ and $\Alph(p) \subseteq \Alph(q)$ or $\Alph(q) \subseteq \Alph(p)$, then $\Alph(u)$ is not contained in $\Alph(p) \cup \Alph(q)$. Moreover, $w$ is not rich.
\end{lemma}
\begin{proof} As in the proof of Lemma~\ref{L:sep-by-x}, we first note that $|w| > |\Alph(w)|$ because if $|w| = |\Alph(w)|$, then $w=v$ and $v$ would be a product of mutually distinct letters which does not take the assumed form. Hence $K_v > 1$ and $H_v > 1$. We must also have $|\Alph(v)| \geq 3$, because if $|\Alph(v)| = 1$ then $v$ would be a palindrome (i.e., $v=p=q$), and if $|\Alph(v)| = 2$, then $p$ and $q$ cannot be separated by a non-empty word in $v$ by Proposition~\ref{P:binary-hearts}.

Without loss of generality, we will assume that $\Alph(q) \subseteq \Alph(p)$. Suppose on the contrary that $v$ is a heart of a GT-word of minimal length such that $v=puq$ where $|u|\geq 2$ and the alphabet of $u$ is contained in $\Alph(p) \cup \Alph(q)$ (where $\Alph(q) \subseteq \Alph(p)$). 

Using very similar reasoning as in the proof of Proposition~\ref{P:binary-hearts} (and keeping in mind that any left/right special factor of $v$ has left/right valence $2$ by Propositions~\ref{P:right-special-characterization} and \ref{P:left-special-characterization}), we deduce that $p \ne q$ and that both $p$ and $q$ are unioccurrent in $v$. Hence $H_v \leq |p|$ and $K_v \leq |q|$. Moreover, one can also show (using much the same reasoning as in the proof of Proposition~\ref{P:binary-hearts}) that $R_v \leq H_v$ and $L_v \leq K_v$, from which it follows that $R_v=L_v$ and $K_v=H_v$ since $\max\{R_v, K_v\} = \max\{L_v, H_v\}$ and $\min\{R_v,K_v\} = \min\{L_v, H_v\}$ by \cite[Corollary 4.1]{aD99onth} and Proposition~\ref{P:RKLH}. This implies that $|pu| \leq R_v + |\Alph(v)| - 2$ since $|pu| = |v| - |q|$ where $|v| = R_v + K_v + |\Alph(v)| - 2$ (by Theorem~\ref{T:characterization1-take2}) and $K_v \leq |q|$. Moreover, since $\Alph(u) \subseteq \Alph(p)$, we have $K_{pu} > 1$ and $H_{pu} > 1$, and therefore $pu$ is (the heart of) a GT-word (by Proposition~\ref{P:factors-closed}). Hence $|pu| = L_{pu} + H_{pu} + \Alph(pu) - 2$ (by Corollary~\ref{C:characterization1-take2}), where $\Alph(pu) = \Alph(v)$ and $H_{pu} = H_v = K_v$ since $H_v \leq |p|$. Therefore
\[
|pu| = L_{pu} + H_v + |\Alph(v)| - 2 \leq R_v + |\Alph(v)| - 2,
\]
i.e., $L_{pu} + H_v \leq R_v$. But $R_v \leq H_v$. This implies that $R_v = H_v$ and $L_{pu} = 0$, which is impossible. 

Thus $\Alph(u)$ is not contained in $\Alph(p) \cup \Alph(q)$; in particular, $u$ contains a letter not in $\Alph(p) \cup \Alph(q)$. Suppose that $u$ also contains a letter in common with $p$ or $q$. Then, using the fact that $v$ contains at most one left (resp.~right) special factor of each length, with any such factor having left (resp.~right) valence~$2$ (Propositions~\ref{P:right-special-characterization} and \ref{P:left-special-characterization}), it can be (tediously) shown that $v$ must take one of the following forms (\textit{cf.} arguments in the proofs of Lemmas~\ref{L:GT-rich} and \ref{L:sep-by-x}).
\begin{itemize}
\item \textbf{Type 1:} $v=aZ_1aZ_2a$ where $a$ is a letter, each $Z_i$ is a product of at least two mutually distinct letters different from $a$ (i.e., $|Z_i| = |\Alph(Z_i)| \geq 2$ and $a \not\in \Alph(Z_i)$ for $i=1,2$), and $\Alph(Z_1) \cap \Alph(Z_2) = \emptyset$.
\item \textbf{Type 2:} $v=(ab)^{m+1}Z(ab)^{n-1}a$ where $a$, $b$ are distinct letters, $m, n \in \NN^+$,  and $Z$ is a product of mutually distinct letters (possibly just a single letter) different from $a$ and $b$ (i.e., $|Z| = |\Alph(Z)|$ and $a,b \not\in \Alph(Z)$).
\item \textbf{Type 3:} The reverse of Type 2.
\end{itemize}
Note that hearts $v$ of Type 1 have $p=q=a$ and $u=Z_1aZ_2$ (e.g., $abcadea$); hearts of Type~2 have $p=(ab)^{m}a$, $u=bZ$, $q=(ab)^{n-1}a$ (e.g., $ababcaba$); and hearts of Type~3 have $p=(ab)^{n-1}a$, $u=Zb$, $q=(ab)^{m}a$ (e.g., $adcbaba$). It is easy to see that any such word is not rich. Indeed, if $v$ is of Type~1, then $v$ contains two non-palindromic complete returns to the letter $a$ (namely $aZ_1a$ and $aZ_2a$). On the other hand, if $v$ is of Type~2 (or similarly Type~3), then $v$ contains the non-palindromic complete return $abZa$ (or $aZba$) to the letter $a$.

It remains to show that if $u$ has no letters in common with $p$ and $q$, then $v$ (and hence $w$) is not rich. We distinguish two cases, as follows.

\textbf{Case 1:} \textit{$u$ is not a palindrome}. Since $\Alph(q) \subseteq \Alph(p)$, the first letter $a$ of $q$ is a factor of $p$ and we can write $p=p'ap''$ and $q=aq'$ where $p''$ is a (possibly empty) word not containing the letter $a$. That is, $v=p'ap''uaq'$ where $ap''ua$ is a non-palindromic complete return to the letter $a$. Thus $v$ is not rich.

\textbf{Case 2:} \textit{$u$ is a palindrome}. If $p=q$ then $v$ ($=pup$) would be a palindrome. But this is absurd because if $v$ were a palindrome, then we would necessarily have $v=p=q$ by definition of $p$ and $q$. Therefore $p \ne q$. Let $x$ denote the last letter of $u$. Since $|u|\geq 2$, there exists a letter $y$ (possibly equal to $x$) such that $xy$ and its reversal $yx$ are factors of $u$. Moreover, if $\ell$ is the last letter of $p$ and $f$ is the first letter of $q$, then $\ell x$ and $xf$ are factors of $v$ where $\{\ell, f\} \cap \{x,y\} = \emptyset$ since $u$ contains no letters in common with $p$ and $q$. So the letter $x$ is a bispecial factor of $v$ of length $1$. This means that neither $p$ nor $q$ contains a left or right special factor, otherwise $v$ would contain another left or right special factor different from $x$, which is not possible by Propositions~\ref{P:right-special-characterization} and \ref{P:left-special-characterization}. So either $p = a^m$ for some letter $a$ and $m \in \mathbb{N}^+$ or $p=(ab)^ma$ where $a$, $b$ are distinct letters and $m \in \mathbb{N}^+$. Likewise, since $q$ has no left or right special factor of length $1$ and $\Alph(q) \subseteq \Alph(p)$, it follows that $q$ must take the form $a^n$, $(ab)^na$ or $(ba)^nb$ for some $n \in \mathbb{N}^+$. If $p=a^m$, then $q=a^n$ where $n\ne m$ (since $p \ne q$ and $\Alph(q) \subseteq \Alph(p)$). However, if this was the case, then $v$ would begin (or end) with a palindrome of the form $a^kua^k$ longer than $p$ (or $q$). So $p$ must take the form $(ab)^ma$ for some $m \in \mathbb{N}^+$. If $q$ also takes the form $(ab)^na$ for some $n \in \mathbb{N}^+$, then $v$ would begin (or end) with a palindrome of the form $(ab)^kau(ab)^ka$ longer than $p$ (or $q$). So $q$ must take the form $(ba)^nb$. But then $v=(ab)^mau(ba)^mb$ which contains the non-palindromic complete return $baub$ to the letter $b$ (and also the non-palindromic complete return $auba$ to the letter~$a$). Thus $v$ is not rich.
\end{proof}

\pagebreak

\begin{lemma} \label{L:rich-disjoint}
Let $w$ be a GT-word with heart $v$ having longest palindromic prefix $p$ and longest palindromic suffix $q$. Suppose $v=puq$ where $u$ is a non-empty word and $\Alph(p) \cap \Alph(q) = \emptyset$. Then $v$ is rich if and only if there exist words $u_1$, $u_2$, and $Z$ (at least one of which is non-empty) such that $u=u_1Zu_2$ where $\Alph(u_1) \subseteq \Alph(p)$, $\Alph(u_2) \subseteq \Alph(q)$, and $Z$ contains no letters in common with $p$ and $q$. 
Moreover, $|Z|=|\Alph(Z)|$ and there exist mutually distinct letters $a$, $b$, $x$, $y$ and positive integers $m, n$ such that one of the following conditions is satisfied:
\begin{itemize}
\item[(i)] $p= (ab)^ma$, $u=Zx$, $q=(yx)^ny$;
\item[(ii)] $p= (ba)^mb$, $u=aZx$, $q=(yx)^ny$;
\item[(iii)] $p= a^{m+1}$, $u=Zx$, $q=(yx)^ny$;
\item[(iv)] $p = (ab)^{m}a$, $u = Z$, $q = x^{n+1}$;
\item[(v)] $p=(ab)^ma$, $u =Z$, $q = (xy)^nx$;
\item[(vi)] $p = a^{m}$, $u = Z$, $q = x^{n}$;
\item[(vii)] the forms of $p$ and $q$ are the opposite of those satisfying one of the conditions \emph{(i)--(vi)} with the reversal of the corresponding $u$ separating them.
\end{itemize}
\end{lemma}
\begin{note} The word $Z$ must be non-empty in conditions (iv)--(vi). 
\end{note}
\begin{proof}  Clearly we have $|\Alph(v)| \geq 3$, because if $|\Alph(v)| = 1$ then $v$ would be a palindrome (i.e., $v=p=q$), and if $|\Alph(v)| = 2$, then $p$ and $q$ cannot be separated by a non-empty word in $v$ by Proposition~\ref{P:binary-hearts}. We also observe that if $|w|=|\Alph(w)|$, then $w=v$ and such a word is simply a product of mutually distinct letters with its first and last letters being its longest palindromic prefix and longest palindromic suffix, respectively. Such a word $v$ has $p$, $u$, and $q$ satisfying condition (vi) with $m=n=1$. So we will now assume that $|w| > |\Alph(w)|$, in which case the heart $v$ of $w$ satisfies $K_v > 1$ and $H_v > 1$. This means, in particular, that each letter in $\Alph(v)$ is both left and right extendable in $v$.

($\Rightarrow:$) Suppose $v$ is rich. Then we must have $|p| \geq 2$ and $|q| \geq 2$. Indeed, if $p=a$ for some letter $a$, then since $H_v > 1$, $p$ occurs at least twice in $v$, and therefore $v$ begins with a palindromic complete return to $p$ that is longer than $p$ (its longest palindromic prefix), which is not possible. Likewise, we deduce that $|q|\geq 2$. We distinguish two cases, as follows.

\textbf{Case 1:} $\Alph(u) \subseteq \Alph(p) \cup \Alph(q)$. Suppose $x$ is the left-most letter in $u$ that is not in the alphabet of $p$. Then $v= pu_1xu_2q$ where $u_1$, $u_2$ are (possibly empty) words with $\Alph(pu_1) = \Alph(p)$. Let $a$ denote the last letter of $pu_1$. Then $a$ occurs in $p$ (since $\Alph(pu_1) = \Alph(p)$), and moreover, since $|p| \geq 2$, there exists a letter $b \ne x$ such that $ba$ and $ab$ are factors of $p$. Since $ax$ is also a factor of $v$, we see that $a$ is a right special factor of $v$ of length $1$. Moreover, the letter $x$ is left special in $v$ because $x \in \Alph(q)$ and so there exists a letter $y \in \Alph(q)$ such that $xy$ and also $yx$ is a factor of $q$; whence $ax$ and $yx$ are factors of $v$ where $y\ne a$. We thus deduce that $p$ contains no left or right special factor of length $1$. Indeed, the right special letter $a$ can only be followed by $b$ (not $x$) in $p$; otherwise $a$ would have right-valence greater than $2$, which implies that $C_v(2) \geq C_v(1) + 2$ (since each letter in $\Alph(v)$ is right-extendable in $v$). But this is impossible because $v$ is a GT-word (also see Proposition~\ref{P:right-special-characterization}). We also note that the letter $a$ can be preceded by only the letter $b$ in $p$; otherwise $a$ would be another left special factor of $v$ of length $1$ (in addition to the letter $x$) which is not possible by Proposition~\ref{P:left-special-characterization}. So $pu_1$ takes one of the following forms: 
\[
\mbox{$pu_1 = a^{m+1}$ ($a=b$, $u_1 = \empt$), \quad $pu_1 = (ab)^ma$ ($a \ne b$, $u_1 = \empt$), \quad or \quad $pu_1 = (ba)^mba$ ($a\ne b$, $u_1 = a$)}
\]
for some $m \in \mathbb{N}^+$. Furthermore, we observe that the letters $a$ and $b$ cannot occur to the right of $x$ in $v$; otherwise one of those letters would be left special. But the only letter that can be left special in $v$ is $x$, by the fact that $|v| = L_v + H_v + |\Alph(v)| - 2$ (by Corollary~\ref{C:characterization1-take2}) together with Proposition~\ref{P:left-special-characterization}. Hence $\Alph(xu_2q) = \Alph(q)$, and using similar reasoning as for $pu_1$, we observe that $xu_2q$ has no left or right special factor of length $1$. Therefore $xu_2q=x(yx)^ny$ for some $n \in \mathbb{N}^+$ where $x \ne y$ (and $u_2 = \empt$). We have thus shown that $v$ satisfies one of the conditions (i)--(iii) with $Z=\empt$, and the reversal of any such $v$ (namely, $\rev{v} = q\rev{u}p$, which is also a GT-word by Proposition~\ref{P:reversal-closure}) satisfies condition (vii).

\textbf{Case 2:} $\Alph(u)$ contains a letter not in $\Alph(p) \cup \Alph(q)$. Using very similar reasoning as in Case~1, we deduce that $v$ satisfies one of the conditions (i)--(vi) where $Z$ is a non-empty word containing no letters in common with $p$ and $q$, and the reversal of any such $v$ (namely, $\rev{v} = q\rev{u}p$, which is also a GT-word by Proposition~\ref{P:reversal-closure}) satisfies condition (vii). Moreover, we note that $|Z|=|\Alph(Z)|$, i.e., $Z$ is a product of mutually distinct letters. Indeed, $Z$ has no left or right special factor of length $1$ because the letters $a$ and $x$ are, respectively, the unique right and left special factors of $v$ of length $1$. This implies that $Z$ is either a product of mutually distinct letters (possibly just a single letter), or $Z$ takes one of the following forms:
\[
 Z= (cd)^kc, \quad Z=(cd)^{k+1}, \quad  \mbox{or} \quad Z=c^{k+1} \quad \mbox{for some $k \in \mathbb{N}^+$},
\]
where $c$ and $d$ are distinct letters. But none of the above forms are possible because in each case the letter $c$ would be a left special factor of $v$ of length $1$ (different from $x$). Thus $Z$ must be a product of mutually distinct letters.

($\Leftarrow$): Conversely, suppose $u=u_1Zu_2$ where $u_1$, $u_2$, and $Z$ are words such that $\Alph(u_1) \subseteq \Alph(p)$, $\Alph(u_2) \subseteq \Alph(q)$, and $Z$ contains no letters in common with $p$ and $q$. Then, using the same reasoning as above, we deduce that $v$ satisfies one of the conditions (i)--(vii) and it is easy to see that every prefix of any such word has a unioccurrent palindromic suffix; whence $v$ is rich.
\end{proof}

\begin{proof}[Proof of Theorem~\ref{T:GT-rich}]  ($\Rightarrow$): Suppose $w$ is rich. Then the heart $v$ of $w$ is clearly rich too since any factor of a rich word is rich (see Lemma~\ref{L:rich-heart}). Suppose, by way of contradiction, that $v$ is the (rich) heart of a rich GT-word $w$ such that $v$ satisfies none of the conditions (i)--(iii) in the statement of the theorem. Then, by Lemmas~\ref{L:GT-rich}, \ref{L:sep-by-x}, and \ref{L:rich-disjoint}, the only possibility is that $v$ takes the form $v=puq$ where $|u| \geq 2$ and $\Alph(p) \subseteq \Alph(q)$ or $\Alph(q) \subseteq \Alph(p)$. However, by Lemma~\ref{L:sep-by-u}, all such $v$ are not rich. Hence one of the conditions (i)--(iii) must hold.

($\Leftarrow$): Conversely, suppose that $v$ satisfies one of the conditions (i)--(iii) in the statement of the theorem. To prove that $w$ is rich, it suffices to prove that $v$ is rich (by Lemma~\ref{L:rich-heart}). We first observe that if $v$ satisfies either of the conditions (ii) or (iii), then $v$ is rich (by Lemmas~\ref{L:GT-rich}, \ref{L:sep-by-x}, and \ref{L:rich-disjoint}). So let us now assume that $v$ satisfies condition (i) and suppose, on the contrary, that $v$ is a heart of a GT-word of minimal length satisfying that condition but $v$ is not rich. Then, by Proposition~\ref{P:aGjJsWlZ09pali}, $v$ contains a non-palindromic complete return to some non-empty palindrome $P$, say $r = PUP$ where $U$ is a non-palindromic word. If $r=PUP$ is a factor of $p$, then $p$ would be a non-rich (heart of a) GT-word shorter than $v$ satisfying condition (i), contradicting the minimality of $v$ as a heart of a GT-word of shortest length satisfying condition (i), but not rich. So $r=PUP$ is not a factor of $p$, and likewise, $r$ is not a factor of $q$. Therefore $r=PUP$ is not wholly contained in either of the palindromes $p$ and $q$. Hence, we have $v=p'rq' = p'PUPq'$ where $p'$ is prefix of $p$ and $q'$ is a suffix of $q$. Note that at least one of the words $p'$ and $q'$ is non-empty; otherwise $v=PUP$ where $P=p=q$ and $|U| \geq 2$, in which case $v$ does not satisfy either of the conditions (i) or (ii). We distinguish three cases, as follows.

\textbf{Case 1:} \textit{The first occurrence of $P$ in $r$ is not contained in $p$.} In this case, $q=P'UPq'$ where $P'$ is a (possibly empty) suffix of $P$ with $P' \ne P$ (i.e., $P'$ is a \textit{proper} suffix of $P$)  because, by assumption, $q$ does not contain $r=PUP$ as a factor. Let $S=PUPq'$. This proper suffix of $v$ is (the heart of) a GT-word (by Proposition~\ref{P:factors-closed}) and its longest palindromic prefix is $P$. Indeed, if $S$ had a palindromic prefix longer than $P$, but shorter than $r=PUP$, then $P$ would occur more than twice in $r$ (i.e., $P$ would occur as an interior factor of $r$), which is impossible because $r$ is a complete return to $P$. On the other hand, if $S$ had a palindromic prefix $Z$ longer than $r=PUP$, then the palindrome $Z$ (which begins with $r=PUP$) would be a non-rich (heart of a) GT-word shorter than $v$ satisfying condition~(i), contradicting the minimality of $v$. So we see that the longest palindromic prefix of $S$ (namely $P$) and the longest palindromic suffix of $S$ (namely $q$) are unseparated in $S$. Hence $S$ is a non-rich (heart of a) GT-word (shorter than $v$) satisfying condition (i), contradicting the minimality of $v$ again. 

\textbf{Case 2:} \textit{The second occurrence of $P$ in $r$ is not contained in $q$.} Using the same reasoning as in Case~1 we reach a contradiction.

\textbf{Case 3:}  \textit{The first occurrence of $P$ in $r$ is contained in $p$ and the second occurrence of $P$ in $r$ is contained in $q$}. In this case, $p$ ends with $PU'$ where $U'$ is a proper prefix of $UP$ and $q$ begins with $U''P$ where $U''$ is a proper suffix of $PU$. Let $S=PUPq'$ where $q'$ is a suffix of $q$. We observe that $S$ has longest palindromic suffix $q$, and using the same reasoning as in Case~1, we deduce that the longest palindromic prefix of $S$ is $P$. Moreover, $P$ and $q$ must be unseparated in $S$. Otherwise, if $P$ and $q$ were separated by a letter or word of length at least $2$ in $S$, then the same would be true of $p$ and $q$ in $v$, which is impossible because $v$ satisfies condition (i). Hence $S$ is a non-rich (heart of a) GT-word (shorter than $v$) satisfying condition~(i), contradicting the minimality of $v$.

In each of the above three cases we have reached a contradiction. Hence $v$ (and therefore $w$) must be~rich.
\end{proof}

\section{Closing Remarks} \label{S:closing-remarks}

In this paper we have introduced a new class of words that encompass the (original) binary trapezoidal words and have studied combinatorial and structural properties of these so-called generalized trapezoidal words. We have also proved some characterizations of these words and have described all those that are rich in palindromes.  A nice (direct) consequence of our characterization of rich GT-words is that all GT-words with palindromic hearts are rich.  In the case of binary trapezoidal words, we have independently proved that the longest palindromic prefix and longest palindromic suffix of any such word must be ``unseparated'', from which it follows that all binary trapezoidal words are rich, originally proved in a different way in \cite{aDaGlZ08rich}. Further work-in-progress concerns the enumeration of generalized trapezoidal words, as has recently been considered in the binary case~\cite{mBaDgF13enum} (see also \cite{aH01onlo} in which an enumeration formula for binary trapezoidal words was independently obtained).

\section{Acknowledgements} 

We thank very much the two anonymous referees for their thoughtful comments and suggestions which greatly helped to improve the paper. The first author would also like to thank members of the \textit{Laboratoire MIS} and \textit{Equipe SDMA} at the Universit\'e de Picardie Jules Verne for their kind hospitality during her 1-month visit in mid 2011 when some of the preliminary work on this paper was done. She also gratefully acknowledges the funding provided by the Universit\'e de Picardie Jules Verne for her visit to Amiens.



\begin{thebibliography}{99}

\bibitem{sBsHmNcR04onth} S.~Brlek, S.~Hamel, M.~Nivat, C.~Reutenauer,  On the palindromic complexity of infinite words, \textit{Internat. J. Found. Comput. Sci.} \textbf{15} (2004) 293--306.

\bibitem{mBaDgF13enum} M.~Bucci, A.~De~Luca, G.~Fici, Enumeration and structure of trapezoidal words, \textit{Theoret. Comput. Sci.} \textbf{468} (2013) 12--22.


\bibitem{fD02acom} F.~D'Alessandro, A combinatorial problem on trapezoidal words, \textit{Theoret. Comput. Sci.} \textbf{273} (2002) 11--33.

\bibitem{aD99onth} A.~de Luca, On the combinatorics of finite words, \textit{Theoret. Comput. Sci.} \textbf{218} (1999) 13--39. 

\bibitem{aDaGlZ08rich} A.~de~Luca, A.~Glen, L.Q.~Zamboni, Rich, Sturmian, and trapezoidal words, \textit{Theoret. Comput. Sci.} \textbf{407} (2008) 569--573.

\bibitem{xDjJgP01epis} X.~Droubay, J.~Justin, G.~Pirillo, Episturmian words and some constructions of de Luca and Rauzy, \textit{Theoret. Comput. Sci.} \textbf{255} (2001) 539--553. 

\bibitem{xDgP99pali} X. Droubay, G. Pirillo, Palindromes and Sturmian words, \textit{Theoret. Comput. Sci.} \textbf{223} (1999), p. 73--85.

\bibitem{aGjJsWlZ09pali} A. Glen, J. Justin, S. Widmer, L.Q. Zamboni, Palindromic richness, \textit{European J. Combin.} \textbf{30} (2009) 510--531.

\bibitem{aH01onlo} A. Heinis, On low-complexity bi-infinite words and their factors, \textit{J. Th\'eor. Nombres Bordeaux} \textbf{13} (2001) 421--442.

\bibitem{mL02alge} M.~Lothaire, \textit{Algebraic Combinatorics on Words, Encyclopedia of Mathematics and its Applications}, vol.~90, \textit{Cambridge University Press}, UK, 2002. 

\bibitem{gHmM40symb} M.~Morse, G. A.~Hedlund, Symbolic dynamics II. Sturmian trajectories, \textit{Amer. J. Math.} \textbf{62} (1940) 1--42.

\end{thebibliography}
\end{document}